\documentclass[a4paper,reqno]{amsart}

\usepackage[foot]{amsaddr}
\usepackage[english]{babel}

\usepackage{csquotes}
\usepackage{amsmath,amssymb,amsfonts,amsthm}
\usepackage{tikz}
\tikzset{>=stealth}
\usetikzlibrary{shapes,arrows,shadows}
\usepackage{graphicx}
\usepackage[utf8]{inputenc}
\usepackage{enumerate}

\usepackage[numbers]{natbib}

\usepackage{booktabs}
\usepackage{subcaption}
\usepackage[colorlinks=true,citecolor=blue]{hyperref}

\usepackage{algorithm}
\usepackage[noend]{algpseudocode}
\makeatletter
\def\BState{\State\hskip-\ALG@thistlm}
\makeatother

\theoremstyle{plain}
\newtheorem{theorem}{Theorem}
\newtheorem{lemma}{Lemma}
\newtheorem{corollary}{Corollary}
\newtheorem{proposition}{Proposition}

\newtheorem{conjecture}{Conjecture}

\theoremstyle{remark}
\newtheorem{remark}{Remark}

\theoremstyle{definition}
\newtheorem{definition}{Definition}
\newtheorem{example}{Example}

\newcommand{\reals}{\mathbb{R}}

\newcommand{\vect}[1]{\mbox{\boldmath$#1$}}

\newcommand{\fBeta}{\mathcal{F}_\beta}
\newcommand{\hBeta}{\mathcal{H}_\beta}
\newcommand{\whBeta}{\mathcal{WH}_\beta}

\newcommand{\supportgraph}{\mathcal{S}}

\newcommand{\ini}{\mathcal{N}^{-}(i)}
\newcommand{\outi}{\mathcal{N}^{+}(i)}
\newcommand{\inone}{\mathcal{N}^{-}(1)}
\newcommand{\outone}{\mathcal{N}^{+}(1)}
\newcommand{\nodeset}{\ensuremath{V}}
\newcommand{\arcset}{\ensuremath{E}}
\DeclareMathOperator{\rank}{rank}

\DeclareMathOperator{\indeg}{deg^{in}}
\DeclareMathOperator{\outdeg}{deg^{out}}

\newcommand{\expect}[1]{\textbf{E}\left(#1\right)}

\let\leq\leqslant
\let\geq\geqslant

\title{Hamiltonian cycles and subsets of discounted occupational measures}
\author[A. Eshragh]{Ali Eshragh$^1$}
\address{$^{1}$School of Mathematical and Physical Sciences, University of Newcastle, NSW,
  Australia}
\author[J.A. Filar]{Jerzy A. Filar$^2$}
\address{$^2$School of Mathematics and Physics, University of Queensland, QLD, Australia}
\thanks{The second author wishes to acknowledge the support from the Australian Research Council
  under grants DP150100618 and DP160101236.}
\author[T. Kalinowski]{Thomas Kalinowski$^{1,3}$}
\address{$^3$School of Science and Technology, University of New England, NSW, Australia}
\author[S. Mohammadian]{Sogol Mohammadian$^1$}
\email[A.~Eshragh]{\{ali.eshragh,sogol.mohammadian\}@newcastle.edu.au}
\email[J.A.~Filar]{j.filar@uq.edu.au}
\email[T.~Kalinowski]{t.kalinow@une.edu.au}

\date{\today}

\begin{document}

\begin{abstract}
  We study a certain polytope arising from embedding the Hamiltonian cycle problem in a discounted
  Markov decision process. The Hamiltonian cycle problem can be reduced to finding particular
  extreme points of a certain polytope associated with the input graph. This polytope is a subset of
  the space of discounted occupational measures. We characterize the feasible bases of the polytope
  for a general input graph $G$, and determine the expected numbers of different types of feasible
  bases when the underlying graph is random. We utilize these results to demonstrate that augmenting
  certain additional constraints to reduce the polyhedral domain can eliminate a large number of
  feasible bases that do not correspond to Hamiltonian cycles. Finally, we develop a random walk
  algorithm on the feasible bases of the reduced polytope and present some numerical results. We
  conclude with a conjecture on the feasible bases of the reduced polytope.  
\end{abstract}


\maketitle

\section{Introduction.}\label{section:intro}
One of the classical problems of combinatorial mathematics is the \emph{Hamiltonian Cycle Problem}
(HCP), named after the Irish mathematician, Sir William Rowan Hamilton. He designed the Icosian
Game. To win this game, a player must visit each of twenty specifically connected cities,
represented by holes on a wooden pegboard, exactly once and return to the starting point. The
Hamiltonian cycle problem is a mathematically generalized version of this game. Given a graph $G$,
the aim is either to find a cycle that passes through every node of $G$ exactly once, or to
determine that no such cycle exists. Cycles that pass through every node of the graph exactly
once are called \emph{Hamiltonian cycles}. If the graph contains at least one Hamiltonian cycle,
then it is called \emph{Hamiltonian}. Otherwise, it is
\emph{non-Hamiltonian}. Figures~\ref{fig:Hamilton_nonHamilton_graph}(\subref{fig:Hamilton_graph})\, and~\ref{fig:Hamilton_nonHamilton_graph}(\subref{fig:nonHamilton_graph}) show
examples of Hamiltonian and non-Hamiltonian graphs on five nodes.

\begin{figure}[htb]
  \begin{subfigure}[b]{.49\textwidth}
    \centering
    \begin{tikzpicture}[xscale=1.5,every node/.style={draw,shape=circle,outer sep=1pt,inner
        sep=1pt,minimum size=.15cm}]
		\node (1) at (0,0) {1};
		\node (2) at (1,0) {2};
		\node (3) at (1.5,-1) {3};
		\node (4) at (.5,-1) {4};
		\node (5) at (-.5,-1) {5};		
		
		\draw[thick,<->] (1) -- (2);
                \draw[thick,<->] (2) -- (3);
                \draw[thick,<->] (3) -- (4);
                \draw[thick,<->] (4) -- (5);
                \draw[thick,<->] (5) -- (1);
                \draw[thick,<->] (1) -- (4);
              \end{tikzpicture}
              \caption{A Hamiltonian graph}\label{fig:Hamilton_graph}
            \end{subfigure}
	\begin{subfigure}[b]{.49\textwidth}
		\centering
		\begin{tikzpicture}[xscale=1.5,every node/.style={draw,shape=circle,outer sep=1pt,inner sep=1pt,minimum size=.15cm}]
		\node (1) at (0,0) {1};
		\node (2) at (1,0) {2};
		\node (3) at (1.5,-1) {3};
		\node (4) at (.5,-1) {4};
		\node (5) at (-.5,-1) {5};

		\draw[thick,<->] (1) -- (2);
		\draw[thick,<->] (2) -- (3);
		\draw[thick,<->] (2) -- (4);
		\draw[thick,<->] (4) -- (5);
		\draw[thick,<->] (5) -- (1);
		\draw[thick,<->] (1) -- (4);		
		\end{tikzpicture}
		\caption{A non-Hamiltonian graph}\label{fig:nonHamilton_graph}
	\end{subfigure}
	\caption{Hamiltonian and non-Hamiltonian graphs}\label{fig:Hamilton_nonHamilton_graph}
\end{figure}
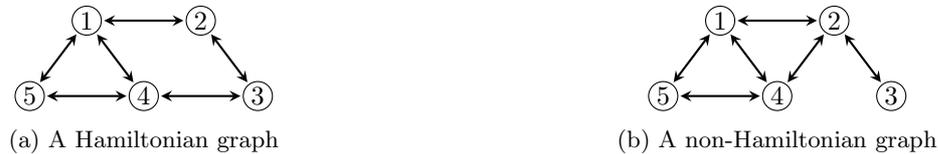

Despite originating in the 1850s, HCP continues to generate a great deal of research interest. The
similarity between HCP and the famous Traveling Salesman Problem (TSP) makes it interesting from a
combinatorial optimization viewpoint. As TSP aims to find a route of minimal distance for a salesman
who starts from a home location, visits every city exactly once and returns to the home location,
HCP can be considered a special case of this problem. To see this, for a given graph $G$, add
artificial arcs $(i,j)$ for every pair $(i,j)$ of distinct nodes which are not connected by an arc in $G$. If we assign distance one to
each original arc and distance two to each artificial arc, the graph $G$ is Hamiltonian if and only if
there is a shortest route in the modified graph with the total distance of $n$.

In 1976, \citet{garey1976planar} showed that HCP is an NP-complete problem.  In particular, no-one
has found a solution algorithm with polynomial worst-case runtime. Its simple appearance, however,
has captivated many researchers, giving rise to a rich literature presenting theoretical and
computational results about this problem. In particular, TSP has been one of the major driving
forces for the development of polyhedral theory for combinatorial optimization. The TSP polytope is
the convex hull of the Hamiltonian cycles of a complete graph, and the great success with solving large
instances of this problem is based on the deep insights into the geometry and combinatorics of this
polytope (see \cite{applegate2011traveling}).

Arguably, one of the central problems in theoretical Operations
Research is TSP. While TSP has received a lot of attention in the
literature, much of its difficulty lies in the ‘visit all nodes
exactly once’ constraint that is captured mathematically in HCP. Since
HCP is known to be NP-complete, probabilistic embedding approaches
such as the line of research continued in the present paper offers a
pathway towards deeper understanding of the essential difficulty of
both HCP and TSP.


In 1994, \citet{filar1994hamiltonian} proposed a new approach to HCP, embedding it into a
\emph{Markov Decision Process} (MDP).  An MDP comprises a state space, an action space, a function
of transition probabilities between states (conditioned on the actions taken by the decision maker)
and a reward function. In the MDP's basic setting, the decision maker takes an action, receives a
reward from the environment, and the environment changes its state. Next, the decision maker
identifies the state of the environment, takes a further action, obtains a reward, and so forth. The
state transitions are probabilistic, and depend solely on the actual state and the action taken by
the decision maker. The reward obtained by the decision maker depends on the action taken, and on
the current state of the environment. The decision maker's actions in each environmental state are
prescribed by a policy.  Markov decision processes are applicable to a wide range of optimization
problems. The model introduced in \citet{filar1994hamiltonian} instigated a new line of research,
which has attracted growing attention (see, for example,
\cite{avrachenkov2016transition,Borkar2011,ejov2004interior,ejov2009refined,ejov2008determinants,Ejov2011,feinberg2000constrained,Filar2015,litvak2009markov}).

In 2000, \citet{feinberg2000constrained} investigated the relationship between HCP and
\emph{discounted} MDPs. In discounted MDPs, a discount factor $\beta \in (0,1)$, which represents
the difference in importance between future and present rewards, is used to discount
rewards. Feinberg showed that HCP can be viewed as a discounted MDP with constraints. Part of his
proof involved the construction of a new polytope corresponding to a given graph $G$, which we shall
refer to as $\fBeta(G)$. The polytope $\fBeta(G)$ is a subset of the space of discounted
occupational measures induced by the constrained discounted MDP. Feinberg showed that if the graph
$G$ is Hamiltonian, the polytope $\fBeta(G)$ has an extreme point, called a \emph{Hamiltonian
  extreme point}, for each of its Hamiltonian cycles. Subsequently, \citet{ejov2009refined}
described some geometric properties of $\fBeta(G)$ and \citet{eshragh2011hybrid} transformed
$\fBeta(G)$ to a polytope $\hBeta(G)$ to improve algorithmic efficiency. In 2011,
\citet{eshragh2011hamiltonian} partitioned all extreme points of $\hBeta(G)$ into five types,
consisting of Hamiltonian extreme points and non-Hamiltonian extreme points of types 1, 2, 3 and 4. They
constructed a new polytope $\whBeta(G)$ by adding new linear constraints, called \emph{wedge
  constraints}, to the polytope $\hBeta(G)$. They showed that, when the discount factor $\beta$ is
sufficiently close to one, the wedge constraints remove the, typically, most abundant
non-Hamiltonian extreme points of types 2,3 and 4, while preserving the Hamiltonian extreme points.

In this paper, we develop geometric properties of the feasible bases of the polytope $\hBeta(G)$ and
characterize them. Moreover, we find the expected number of feasible bases associated with
Hamiltonian extreme points as well as each type of non-Hamiltonian extreme point in the random polytope
$\hBeta(G_{n,p})$ for an input binomial random graph $G_{n,p}$. We show that in expectation, the
feasible bases corresponding to non-Hamiltonian extreme points of Type 4 are the majority. Motivated
by the results of \citet{eshragh2011hamiltonian}, we construct two algorithms based on a simple
random walk on feasible bases of the two polytopes $\hBeta(G)$ and $\whBeta(G)$. We use
these two random walks to compare the numbers of feasible bases associated with Hamiltonian extreme
points in these two polytopes, and explore the efficiency of the wedge constraints. While computational
experiments confirm our analytical results on the feasible bases of the polytope $\hBeta(G)$, they
reveal that the wedge constraints improve algorithmic efficiency for values of the discount
factor $\beta$ sufficiently close to $1$. All these theoretical and computational results support a
new conjecture on the feasible bases of the polytope $\whBeta(G)$ stated at the end of the paper.

The remainder of this paper is organized as follows: In Section~\ref{section:formulation}\,, we
review some preliminary results and introduce the polytope $\hBeta(G)$. In
Section~\ref{section:geometric_proerties}\,, we characterize the feasible bases of $\hBeta(G)$. In
Section~\ref{section:prevelance}\,, we investigate the expected prevalence of feasible bases of
$\hBeta(G_{n,p})$ for an input binomial random graph. More precisely, we find the expected number of each type of 
feasible basis in the random polytope $\hBeta(G_{n,p})$. In
Section~\ref{section:reducing_region}\,, we discuss the results of \cite{eshragh2011hybrid, eshragh2011hamiltonian}
concerning $\whBeta(G)$. In Section~\ref{section:numerical_result}\,, we investigate the
efficiency of this polytope by developing and running two different random walk
algorithms on feasible bases of the polytope $\whBeta(G)$. In
Section~\ref{sec:conclusion}, we conclude with a conjecture on the
feasible bases of the polytope $\whBeta(\bar{G}_{n,p})$, where $\bar{G}_{n,p}$ is a Hamiltonian binomial random graph.
%
%

\section{Formulation of HCP through discounted MDPs.}\label{section:formulation}
Consider a directed graph $G=(\nodeset,\arcset)$, where $\nodeset=\{1,2,\dots,n\}$ is the set of
nodes and $\arcset$ is the set of arcs. Throughout this paper, $G$ refers to such a directed graph
on $n$ nodes, unless otherwise stated. For each node $i\in \nodeset$, we define the in-neighborhood
set, $\ini$, and the out-neighborhood set, $\outi$, of $i$ by
\[\ini=\{j\in \nodeset: (j,i)\in \arcset\}, \hspace*{1cm} \outi=\{j\in \nodeset:(i,j)\in \arcset\}.\]
For a node $i\in\nodeset$, its \emph{in-degree}, denoted by $\indeg(i)$, is the cardinality of
$\ini$, and its \emph{out-degree}, denoted by $\outdeg(i)$, is the cardinality of
$\outi$. The \emph{degree} of $i\in\nodeset$, denoted by $\deg(i)$ is the total number of arcs
incident with $i$, that is, $\deg(i)=\indeg(i)+\outdeg(i)$.

Let us consider a discounted Markov decision process with state space $\{1,2,\dots,n\}$, action
sets $\mathcal{A}(i)$, for each state $i$, and a discount factor $\beta \in (0,1)$. The embedding
of the original graph $G$ in such an MDP is based on a one-to-one correspondence of nodes of $G$
with the states of MDP and actions in state $i$ with arcs emanating from node $i$. That is,
$\mathcal{A}(i) = \outi$ for $i = 1,2,\dots,n$. The Markovian transition probabilities normally
accompanying an MDP are particularly simple in this embedding. Namely, choice of an action
corresponding to arc $(i,j)$ results in the transition from state $i$ to state $j$, with probability
one.

Using the above embedding, \citet{feinberg2000constrained} converted HCP to a constrained
discounted MDP and showed that finding a Hamiltonian cycle is equivalent to finding a structured
extreme point of a certain polytope that we shall call the Feinberg polytope $\fBeta(G)$. This
result is made precise in Theorem~\ref{Fienberg_thm}, below.

\begin{theorem}[\citet{feinberg2000constrained}]\label{Fienberg_thm}
  Consider the embedding of the graph $G=(\nodeset,\arcset)$ in a constrained discounted MDP with a discount factor $\beta$ and the
  polytope $\fBeta(G)$ characterized by
  \begin{align}
    \sum_{j \in \outone} y_{1j} - \beta \sum_{j \in \inone} y_{j1} &= 1,\label{eq_f:flow_conservation_1}\\
    \sum_{j \in \outi} y_{ij} - \beta \sum_{j \in \ini} y_{ji} &= 0 && \textup{for all } i \in \nodeset \setminus \{1\},\label{eq_f:flow_conservation_i}\\
    \sum_{j \in \outone} y_{1j} & = \frac{1}{1 - \beta^n},\label{eq_f:flow_injection}\\
    y_{ij}&\geq 0 &&\textup{for all }(i,j) \in \arcset.\label{eq_f:nonnegativity}
  \end{align} 
  The graph $G$ is Hamiltonian if and only if there exists an extreme point of $\fBeta(G)$ that has
  exactly $n$ positive coordinates tracing out a Hamiltonian cycle in $G$.
\end{theorem}  
In MDP literature the polyhedral domain defined by
constraints~\eqref{eq_f:flow_conservation_1}, \eqref{eq_f:flow_conservation_i} and
\eqref{eq_f:nonnegativity} is called \emph{the space of discounted occupational measures}. These
spaces have been studied extensively (see, for example,
\cite{altman1999constrained,kallenberg1983linear,puterman2014markov}). Indeed,
\citet{feinberg2000constrained} exploited MDP properties of these measures to prove
Theorem~\ref{Fienberg_thm}.

\citet{eshragh2011hybrid} transformed the polytope $\fBeta(G)$ by changing
variables $x_{ij}:= (1-\beta^n)y_{ij}$ for all $(i,j) \in \arcset$ to produce the polytope
$\hBeta(G) \subseteq \reals^{\lvert \arcset \rvert}$ defined by the constraints
\begin{align}
\medskip
\sum_{j\in \outone} x_{1j}-\beta\sum_{j\in \inone}
x_{j1} & = 1-\beta^n,  \label{eq:flow_conservation_1}\\
\sum_{j\in \outi} x_{ij}-\beta\sum_{j\in \ini}
x_{ji} & = 0 &&\mbox{for all}\ i \in \nodeset \setminus\{1\},  \label{eq:flow_conservation_i}\\
\medskip \sum_{j\in \outone}x_{1j} & = 1, \label{eq:flow_injection} \\
x_{ij} & \geq 0 &&\mbox{for all}\ (i,j)\in \arcset. \label{eq:nonnegativity}
\end{align}
Since values of $\beta$ close to one were shown to be important in
\cite{eshragh2011hamiltonian}--\cite{eshragh2011hybrid}, this transformation eliminates numerical
instability in~\eqref{eq_f:flow_injection}. In the remainder of this paper, $A$ and $\vect b$ denote
the constraint matrix and the right-hand side vector of
constraints~\eqref{eq:flow_conservation_1}--\eqref{eq:flow_injection}, respectively. Indeed, $A$
and $\vect b$ depend on the parameter $\beta$. However, for simplicity and because
we do not consider more than one value of $\beta$ at a time, we do not make this dependence on
$\beta$ explicit. The following definition is motivated directly from Theorem~\ref{Fienberg_thm}.
\begin{definition}\label{def:Ham_EP}
	Let $\vect x$ be an extreme point of the polytope $\mathcal{H}_\beta(G)$. If the positive
	coordinates of $\vect x$ trace out a Hamiltonian cycle in the graph $G$, $\vect x$ is called
	a Hamiltonian extreme point. Otherwise, it is called a non-Hamiltonian extreme point.
\end{definition}

As an example, let us construct the polytope $\hBeta(G)$ for the graph given in
Figure~\ref{fig:Hamilton_nonHamilton_graph}(\subref{fig:Hamilton_graph}):
\begin{align*}
x_{12} + x_{14} + x_{15} - \beta (x_{21}+x_{41}+x_{51}) &= 1 - \beta^5\\
x_{21} + x_{23} - \beta (x_{12} + x_{32}) &= 0\\
x_{32} + x_{34} - \beta (x_{23} + x_{43}) &= 0\\
x_{41} + x_{43} + x_{45} - \beta (x_{14} + x_{34} + x_{54}) &= 0\\
x_{51} + x_{54} - \beta (x_{15} + x_{45}) &= 0\\
x_{12} + x_{14} + x_{15} &= 1\\
x_{12},x_{14},x_{15},x_{21},x_{23},x_{32},x_{34},x_{41},x_{43},x_{45},x_{51},x_{54} & \geq 0.
\end{align*}
It can be written in the form $A\vect{x} = \vect{b}$, $\vect{x} \geq\vect 0$ as follows:
\begin{align*}
\begin{pmatrix}
1& 1& 1 &-\beta&0&0&0&-\beta&0&0&-\beta&0 \\
-\beta & 0& 0 & 1 & 1 & -\beta & 0 & 0 & 0  & 0 & 0 & 0\\
0      & 0      & 0 & 0 & -\beta & 1 & 1 & 0 & -\beta & 0  & 0 &0\\
0      & -\beta & 0 & 0 & 0 & 0 & -\beta & 1 & 1  & 1&0&-\beta\\
0      & 0      & -\beta & 0 & 0 & 0 & 0 & 0 & 0  & -\beta&1&1\\
1      & 1      & 1 & 0 & 0 & 0 & 0 & 0 & 0  & 0 & 0 &0    	
\end{pmatrix}
\begin{pmatrix}
x_{12}\\
x_{14}\\
x_{15}\\
x_{21}\\
x_{23}\\
x_{32}\\
x_{34}\\
x_{41}\\
x_{43}\\
x_{45}\\
x_{51}\\
x_{54}
\end{pmatrix}&=
\begin{pmatrix}
1 - \beta ^5\\
0\\
0\\
0\\
0\\
1
\end{pmatrix}\\\medskip
x_{12},x_{14},x_{15},x_{21},x_{23},x_{32},x_{34},x_{41},x_{43},x_{45},x_{51},x_{54} & \geq 0.
\end{align*}
It is easy to see that
\begin{align*}
\begin{cases}
 x_{12} = 1, x_{23} = \beta, x_{34} = \beta^2, x_{45} =\beta^3,x_{51} = \beta^4\\
 x_{14} = x_{15} = x_{21} = x_{32} =x_{41}=x_{43} = x_{54} = 0
\end{cases}
\end{align*}
is an extreme point of $\hBeta(G)$. Furthermore, the set of arcs
$\{ (1,2), (2,3), (3,4), (4,5), (5,1) \}$, corresponding to the non-zero variables, traces out a
Hamiltonian cycle in the graph $G$.

Theorem~\ref{Fienberg_thm} shows that the polytope $\hBeta(G)$ reflects some properties of
the Hamiltonian cycles in $G$ and, thus, can be utilized as the basis for a procedure to search for
Hamiltonian cycles. This procedure is outlined in Algorithm 1.
\begin{algorithm*}
	\caption{Search Algorithm for HCP}
	\begin{algorithmic}[1]
	\State \textbf{Input:} a graph $G$
	\State Search for a Hamiltonian extreme point among the extreme points of $\hBeta(G)$
	\If{search is successful}
	\State output the Hamiltonian extreme point
	\Else
	\State claim that $G$ is not Hamiltonian\label{claim:nonHamiltonian}
	\EndIf	
	\end{algorithmic}\label{alg:HC_search}
\end{algorithm*}

We note that this is only a framework of an algorithm. In order to derive a practical algorithm, it
still needs to be specified how the search in Line 2 of Algorithm~\ref{alg:HC_search} is implemented. Depending on
this implementation, the resulting algorithm can have quite different properties. In particular, if we search by sampling extreme points, we obtain a randomized search algorithm
which, with a certain probability, will claim that the input graph $G$ is non-Hamiltonian, although
in fact it contains a Hamiltonian cycle. Thus, one may try to alleviate this error probability by
designing a clever random search algorithm. The motivation of our work is to better understand the
properties of random walk based sampling methods in this context. Since each extreme point can be
identified with the set of its corresponding feasible bases. Hence, instead of sampling the extreme
points of the polytope $\hBeta(G)$ we can sample its feasible bases. This yields an efficient
algorithm provided two technical conditions are satisfied which can be informally stated as follows.
\begin{enumerate}
\item There are sufficiently many extreme points (or feasible bases) corresponding to Hamiltonian
  cycles, and 
\item The random walk used of sampling converges to the uniform distribution quickly enough.
\end{enumerate}
The first condition ensures that the error probability is small. More precisely, we can bound the
probability that among $t$ uniform samples there is none that corresponds to a Hamiltonian
cycle. The second condition allows us to bound the number of random walk steps to obtain an
approximately uniform sample (see~\cite{jerrum2003counting} for more details and precise
statements).

In order to establish the two conditions listed above for a random walk on the feasible bases of
$\hBeta(G)$, it is crucial to understand the structure of these feasible bases. The study of the
structure of the extreme points of $\hBeta(G)$ has been initiated in
\cite{ejov2009refined,eshragh2011hamiltonian}. However, these references did not consider the
structure of the feasible bases. In Section~\ref{section:geometric_proerties}, we characterize
feasible bases of the polytope $\hBeta(G)$, and describe their structural and geometric properties.


\section{Feasible bases of the polytope $\mathcal{H}_\beta(G)$.}\label{section:geometric_proerties}
In this section, we consider feasible bases of the polytope $\hBeta(G)$. At the outset, we recall
some basic notation and definitions of polytopes, in general.

\begin{definition}
  Let $\mathcal{P} \subseteq \reals^\eta$ be a polytope defined by the constraints
  $M\vect{x} = \vect{v}$, $\vect{x}\geq\vect 0$, where $M$ is a $\kappa\times \eta$ matrix of rank
  $\kappa\leq\eta$, and $\vect v$ is an $\eta\times 1$ column vector. An extreme point (or vertex)
  of $\mathcal{P}$ is a point $\vect{x} \in \mathcal{P}$ with the property that $M$ has a
  nonsingular $\kappa\times \kappa$-submatrix $M_B$ such that the components of $\vect{x}$
  corresponding to columns not in $M_B$ are zero, and those components corresponding to columns in
  $M_B$ are given by $M_B^{-1}\vect{v}$. For a given submatrix $M_B$ and the corresponding extreme
  point $\vect x$, those components of $\vect x$ associated with the columns of $M_B$ are called
  \emph{basic variables}. The set of all $\kappa$ basic variables is called a \emph{feasible
    basis}. Two distinct feasible bases are adjacent if and only if they have exactly $\kappa-1$
  common basic variables. An extreme point is called \emph{degenerate} if it has less than $\kappa$
  non-zero components. Otherwise, that is, if it has exactly $\kappa$ non-zero components, it is
  \emph{non-degenerate}.
\end{definition} 
\begin{remark}
  If all extreme points are non-degenerate, then there exists a one-to-one correspondence between
  feasible bases and extreme points. Otherwise, some extreme points may be associated with more than
  one feasible basis. Indeed, such degeneracy is generic in applications.
\end{remark}

The polytope $\hBeta(G)$, where $\beta \in (0,1)$, is generated by $n+1$ linearly independent
constraints. Accordingly, each feasible basis of $\hBeta(G)$ has $n+1$ basic variables. As the
columns of the constraint matrix $A$ can be identified with the arcs of $G$, characterizing the
feasible bases of $\hBeta(G)$ is equivalent to characterizing the arc sets $B \subseteq \arcset$
with $\lvert B \rvert = n+1$ satisfying the following two conditions:
 
\begin{description}
\item[B1]\label{item:B1} The columns of the $(n+1)\times (n+1)$ matrix $A_B$ are linearly independent, where
  $A_B$ is constructed by choosing those columns of $A$
  corresponding to $B$. 	
\item[B2]\label{item:B2} The inequality $(A_B)^{-1} \vect b \geq \vect 0$ is satisfied,  which is
  equivalent to~(\ref{eq:nonnegativity}).
\end{description}
Henceforth, we use condition \textbf{B1} and `the set $B$ is linearly independent', interchangeably.

We recall some results pertaining to the structure of extreme points of $\hBeta(G)$. For this, we
present some definitions.
\begin{definition}
  Let $\vect{x}$ be an extreme point of $\hBeta(G)$. The \emph{support} of $\vect x$ is defined to
  be the set of its non-zero coordinates. Since the variables $x_{ij}$ correspond to arcs of $G$,
  the support of $\vect x$ can be identified with a subgraph of $G$, which we denote by
  $\supportgraph(G,\vect x)$. More precisely, $\supportgraph (G,\vect x)$ is the graph with node set
  $\nodeset$ and arc set $\{(i,j)\in\arcset\,:\,x_{ij}>0\}$. Since there is no danger of ambiguity,
  we use the term \emph{support of $\vect x$} to refer to the graph $\supportgraph(G,\vect x)$, the
  arc set of this graph, or the set of variables $x_{ij}$ corresponding to these arcs.
\end{definition}
According to Definition~\ref{def:Ham_EP}, an extreme point $\vect{x}$ is a Hamiltonian extreme point
if and only if the support $\supportgraph(G,\vect x)$ is a Hamiltonian cycle in $G$. Furthermore,
\citet{ejov2009refined} showed that, for a non-Hamiltonian extreme point $\vect{x}$,
$\supportgraph(G,\vect x)$ has a specific structure. In order to state their result, which is
summarized in Theorem~\ref{thm:Ejov}, we need to define some important paths and cycles in
$G$.
\begin{definition}[\citet{ejov2009refined}]
  A cycle of the form $1\rightarrow v_1 \rightarrow \dotsb\rightarrow v_k \rightarrow 1$ with
  distinct nodes $v_i$ $(i = 1, \dotsc, k)$ and $k < n-1$ is called a \emph{short cycle}. A path of
  the form $1 \rightarrow v_1 \rightarrow v_2 \rightarrow \dotsb\rightarrow v_k \rightarrow v_j$
  with distinct nodes $v_i$ $(i = 1, \dotsc , k)$ and $1 \leq j < k < n$ is called a \emph{noose
    path}. The cycle $v_j \rightarrow v_{j+1}\rightarrow \dots \rightarrow v_k \rightarrow v_j$ is
  the associated \emph{noose cycle}.
\end{definition}
\begin{example}
  For $n=5$, the arcs $(1,2)$, $(2,3)$ and $(3,1)$ form a short cycle, and the arcs $(1,2)$,
  $(2,3)$, $(3,4)$ and $(4,2)$ form a noose path with the associated noose cycle $\{(2,3),(3,4),(4,2)\}$.
\end{example}
\begin{theorem}[\citet{ejov2009refined}]\label{thm:Ejov}
  If $\vect{x}$ is a non-Hamiltonian extreme point of the polytope $\hBeta(G)$, then the support $\supportgraph(G,\vect{x})$ is the union of a short cycle and a
  noose path. In particular, this implies that the support of a non-Hamiltonian extreme point
  has a unique node of out-degree two, called the ``splitting node'', and all other nodes have
  out-degrees at most one.
\end{theorem}

Figure~\ref{fig:supports} illustrates the supports for a Hamiltonian and
non-Hamiltonian extreme point of the polytope $\hBeta(G)$, where $G$ is the graph shown in
Figure~\ref{fig:Hamilton_nonHamilton_graph}(\subref{fig:Hamilton_graph}). More precisely, while the
support shown in Figure~\ref{fig:supports}(\subref{fig:HamiltonSupport}) corresponds to a
Hamiltonian extreme point with positive components $x_{12}= 1$, $x_{23} = \beta$, $x_{34}=\beta^2$, $x_{45}=\beta^3$ and
$x_{51}= \beta^4$, the support in Figure~\ref{fig:supports}(\subref{fig:nonHamilton_support}) is
associated with a non-Hamiltonian extreme point with positive components $x_{12}= 1- \beta^2$, $x_{14} = \beta^2$,
$x_{23}=\beta$, $x_{32}=\beta^2$, $x_{45}=\beta^3$ and $x_{51}= \beta^4$. In the latter graph, the
arc sets $\{(1,4), (4,5), (5,1)\}$, $\{(1,2),(2,3), (3,2)\}$ and $\{(2,3), (3,2)\}$ form the short
cycle, noose path and noose cycle, respectively. Furthermore, node $1$ is the splitting node for
this extreme point.

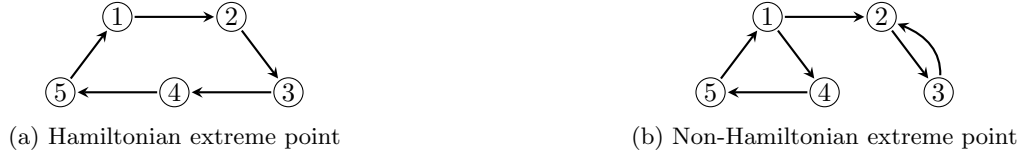
\begin{figure}
	\begin{subfigure}[b]{.49\textwidth}
		\centering
		\begin{tikzpicture}[xscale=1.5,every node/.style={draw,shape=circle,outer sep=1pt,inner sep=1pt,minimum size=.15cm}]
		\node (1) at (0,0) {1};
		\node (2) at (1,0) {2};
		\node (3) at (1.5,-1) {3};
		\node (4) at (.5,-1) {4};
		\node (5) at (-.5,-1) {5};

		\draw[thick,->] (1) -- (2);
		\draw[thick,->] (2) -- (3);
		\draw[thick,->] (3) -- (4);
		\draw[thick,->] (4) -- (5);
		\draw[thick,->] (5) -- (1);
		
		\end{tikzpicture}
		\caption{Hamiltonian extreme point}\label{fig:HamiltonSupport}
	\end{subfigure}
	\hfill
	\begin{subfigure}[b]{.49\textwidth}
		\centering
	\begin{tikzpicture}[xscale=1.5,every node/.style={draw,shape=circle,outer sep=1pt,inner sep=1pt,minimum size=.15cm}]
		\node (1) at (0,0) {1};
		\node (2) at (1,0) {2};
		\node (3) at (1.5,-1) {3};
		\node (4) at (.5,-1) {4};
		\node (5) at (-.5,-1) {5};

		\draw[thick,->] (1) -- (2);
		\draw[thick,->] (2) -- (3);
		\draw[thick,bend right=30,->] (3) to (2);
		\draw[thick,->] (1) -- (4);
		\draw[thick,->] (4) -- (5);
		\draw[thick,->] (5) -- (1);
		\end{tikzpicture}
		\caption{Non-Hamiltonian extreme point}\label{fig:nonHamilton_support}
	\end{subfigure}
	\caption{Supports of a Hamiltonian and a non-Hamiltonian extreme point for the graph in Figure~\ref{fig:Hamilton_nonHamilton_graph}(\subref{fig:Hamilton_graph}).}\label{fig:supports}
\end{figure}

Motivated by Theorem~\ref{thm:Ejov}, \citet{eshragh2011hamiltonian} partitioned the set of
non-Hamiltonian extreme points of the polytope $\hBeta(G)$ into four types, based on their
supports. More precisely, for a given non-Hamiltonian extreme point, if in the corresponding support:
\begin{description}
\item[Type 1] every node has in-degree at least one, the short cycle and the noose cycle are
  node-disjoint, and there is one arc connecting the splitting node (which lies on the short
  cycle) to a node on the noose cycle (Figure~\ref{fig:types}(\subref{fig:nonHamilton_type1})). 
\item[Type 2] every node has in-degree at least one, the short cycle and the noose cycle are
  node-disjoint, and they are linked by a path of length at least two, connecting the splitting node
  with a node on the noose cycle (Figure~\ref{fig:types}(\subref{fig:nonHamilton_type2})).
\item[Type 3] every node has in-degree at least one, and the short cycle and noose cycle have at
  least one node in common (Figure~\ref{fig:types}(\subref{fig:nonHamilton_type3})).
\item[Type 4] at least one node has degree zero (Figure~\ref{fig:types}(\subref{fig:nonHamilton_type4})).
\end{description}

\begin{remark}\label{rem:support_of_type4}
It follows from Theorem~\ref{thm:Ejov} that the support of each non-Hamiltonian extreme point
of Type 4 can be represented as a support of any of types 1--3 on less then $n$ nodes or a
short cycle on less than $n$ nodes with an extra arc $(i,j)$, where nodes $i$ and $j$ are on the short
cycle and node $i$ comes after node $j$.	
\end{remark}
  
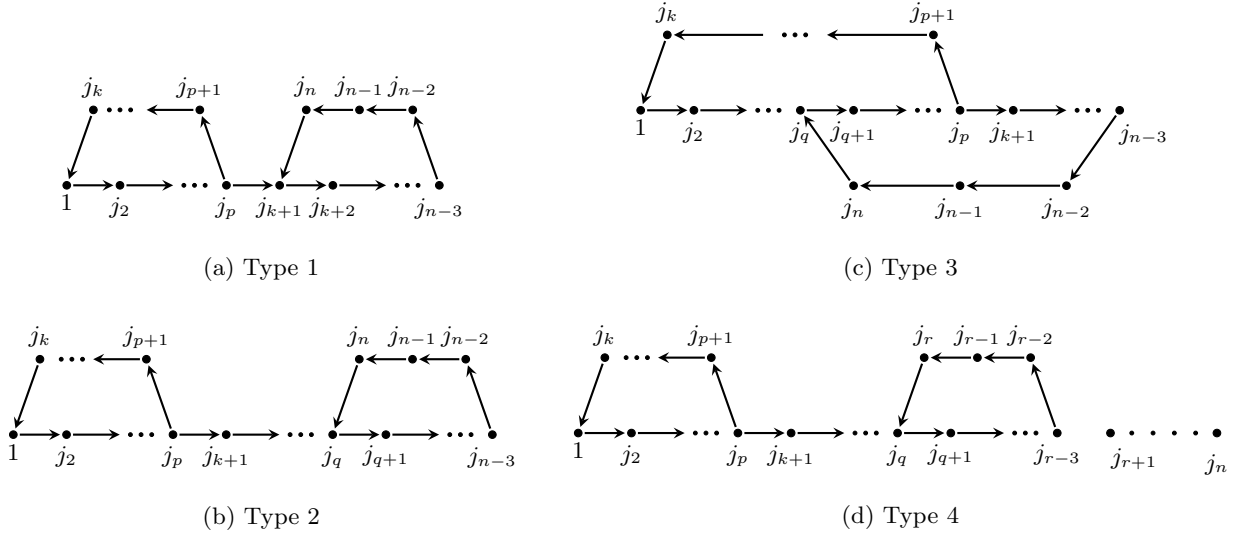
\begin{figure}
  \centering
  {\small
    \begin{minipage}[b]{.45\textwidth}
    \begin{subfigure}{\linewidth}
	\centering
	\begin{tikzpicture}[xscale=1.4,every node/.style={draw,fill=black,shape=circle,outer sep=1pt,inner sep=1pt,minimum size=.1cm}]
	
	\node (1)  at (0,0)    [label={[label distance=-2pt]-90:$1$}]{};
	\node (2)  at (.5,0)   [label={[label distance=-2pt]-90:$j_2$}]{};
	\node (3)  at (1.5,0)  [label={[label distance=-2pt]-90:$j_p$}]{};
	\node (4)  at (2,0)    [label={[label distance=-6pt]-90:$j_{k+1}$}]{};
	\node (5)  at (2.5,0)    [label={[label distance=-6pt]-90:$j_{k+2}$}]{};
	\node (7)  at (3.5,0)  [label={[label distance=-6pt]-90:$j_{n-3}$}]{};
	\node (8)  at (.25,1)  [label={[label distance=-2pt]90:$j_k$}]{};
	\node (9)  at (1.25,1) [label={[label distance=-6pt]90:$j_{p+1}$}]{};
	\node (10) at (2.25,1) [label={[label distance=-2pt]90:$j_n$}]{};
	\node (11) at (2.75,1) [label={[label distance=-6pt]90:$j_{n-1}$}]{};
	\node (12) at (3.25,1) [label={[label distance=-6pt]90:$j_{n-2}$}]{};
	
	\draw[thick,->] (1) -- (2);
	\draw[thick,->] (2) -- (1,0);
	\draw[thick,->] (3) -- (4);
	\draw[thick,->] (4) -- (5);
	\draw[thick,->] (5) -- (3,0);
	\draw[thick,->] (8) -- (1);
	\draw[thick,->] (9) -- (.75,1);
	\draw[thick,->] (7) -- (12);
	\draw[thick,->] (12) -- (11);
	\draw[thick,->] (11) -- (10);
	\draw[thick,->] (10) -- (4);
	\draw[thick,->] (3) -- (9);

	\node[outer sep=.1pt,inner sep=.5pt,minimum size=.01cm] at (1.1,0) {};
	\node[outer sep=.1pt,inner sep=.5pt,minimum size=.01cm] at (1.2,0) {};
	\node[outer sep=.1pt,inner sep=.5pt,minimum size=.01cm] at (1.3,0) {};
	
	\node[outer sep=.1pt,inner sep=.5pt,minimum size=.01cm] at (3.1,0) {};
	\node[outer sep=.1pt,inner sep=.5pt,minimum size=.01cm] at (3.2,0) {};
	\node[outer sep=.1pt,inner sep=.5pt,minimum size=.01cm] at (3.3,0) {};
	\node[outer sep=.1pt,inner sep=.5pt,minimum size=.01cm] at (.4,1) {};
	\node[outer sep=.1pt,inner sep=.5pt,minimum size=.01cm] at (.5,1) {};
	\node[outer sep=.1pt,inner sep=.5pt,minimum size=.01cm] at (.6,1) {};
      \end{tikzpicture}
      \caption{Type 1}\label{fig:nonHamilton_type1}		
    \end{subfigure}
    
      \bigskip
		\begin{subfigure}{\linewidth}
			\centering
		\begin{tikzpicture}[xscale=1.4,every node/.style={draw,fill=black,shape=circle,outer sep=1pt,inner sep=1pt,minimum size=.1cm}]
		
		\node (1)  at (0,0)    [label={[label distance=-2pt]-90:$1$}]{};
		\node (2)  at (.5,0)   [label={[label distance=-2pt]-90:$j_2$}]{};
		\node (3)  at (1.5,0)  [label={[label distance=-2pt]-90:$j_p$}]{};
		\node (4)  at (2,0)    [label={[label distance=-6pt]-90:$j_{k+1}$}]{};
		\node (5)  at (3,0)    [label={[label distance=-2pt]-90:$j_q$}]{};
		\node (6)  at (3.5,0)  [label={[label distance=-6pt]-90:$j_{q+1}$}]{};
		\node (7)  at (4.5,0)  [label={[label distance=-6pt]-90:$j_{n-3}$}]{};
		\node (8)  at (.25,1)  [label={[label distance=-2pt]90:$j_k$}]{};
		\node (9)  at (1.25,1) [label={[label distance=-6pt]90:$j_{p+1}$}]{};
		\node (10) at (3.25,1) [label={[label distance=-2pt]90:$j_n$}]{};
		\node (11) at (3.75,1) [label={[label distance=-6pt]90:$j_{n-1}$}]{};
		\node (12) at (4.25,1) [label={[label distance=-6pt]90:$j_{n-2}$}]{};
		
		\draw[thick,->] (1) -- (2);
		\draw[thick,->] (2) -- (1,0);
		\draw[thick,->] (3) -- (4);
		\draw[thick,->] (4) -- (2.5,0);
		\draw[thick,->] (5) -- (6);
		\draw[thick,->] (6) -- (4,0);
		\draw[thick,->] (8) -- (1);
		\draw[thick,->] (9) -- (.75,1);
		\draw[thick,->] (7) -- (12);
		\draw[thick,->] (12) -- (11);
		\draw[thick,->] (11) -- (10);
		\draw[thick,->] (10) -- (5);
		\draw[thick,->] (3) -- (9);
		
		\node[outer sep=.1pt,inner sep=.5pt,minimum size=.01cm] at (1.1,0) {};
		\node[outer sep=.1pt,inner sep=.5pt,minimum size=.01cm] at (1.2,0) {};
		\node[outer sep=.1pt,inner sep=.5pt,minimum size=.01cm] at (1.3,0) {};
		\node[outer sep=.1pt,inner sep=.5pt,minimum size=.01cm] at (2.6,0) {};
		\node[outer sep=.1pt,inner sep=.5pt,minimum size=.01cm] at (2.7,0) {};
		\node[outer sep=.1pt,inner sep=.5pt,minimum size=.01cm] at (2.8,0) {};
		\node[outer sep=.1pt,inner sep=.5pt,minimum size=.01cm] at (4.1,0) {};
		\node[outer sep=.1pt,inner sep=.5pt,minimum size=.01cm] at (4.2,0) {};
		\node[outer sep=.1pt,inner sep=.5pt,minimum size=.01cm] at (4.3,0) {};
		\node[outer sep=.1pt,inner sep=.5pt,minimum size=.01cm] at (.45,1) {};
		\node[outer sep=.1pt,inner sep=.5pt,minimum size=.01cm] at (.55,1) {};
		\node[outer sep=.1pt,inner sep=.5pt,minimum size=.01cm] at (.65,1) {};
              \end{tikzpicture}
              \caption{Type 2}\label{fig:nonHamilton_type2}	
              \end{subfigure}
            \end{minipage}
            \begin{minipage}[b]{.54\textwidth}
		\begin{subfigure}{\linewidth}
			\centering
			\begin{tikzpicture}
			[xscale=1.4,every node/.style={draw,fill=black,shape=circle,outer sep=1pt,inner sep=1pt,minimum size=.1cm}]
			
			\node (0) at (0,0)     [label={[label distance=-2pt]-90:$1$}]{};
			\node (1) at (.5,0)    [label={[label distance=-2pt]-90:$j_2$}]{};
			\node (2) at (1.5,0)   [label={[label distance=-2pt]-90:$j_q$}]{};
			\node (3) at (2,0)     [label={[label distance=-6pt]-90:$j_{q+1}$}]{};
			\node (4) at (3,0)     [label={[label distance=-2pt]-90:$j_p$}]{};
			\node (5) at (3.5,0)   [label={[label distance=-6pt]-90:$j_{k+1}$}]{};
			\node (6) at (4.5,0)   [label={[label distance=-2pt]-60:$j_{n-3}$}]{};

			\draw[thick,->] (0) -- (1);	
			\draw[thick,->] (1) -- (1,0);
			\draw[thick,->] (2) -- (3);
			\draw[thick,->] (3) -- (2.5,0);	
			\draw[thick,->] (4) -- (5);
			\draw[thick,->] (5) -- (4,0);

			\node (7) at (.25,1)      [label={[label distance=-2pt]90:$j_k$}]{};	
			\node (11) at (2.75,1)    [label={[label distance=-6pt]90:$j_{p+1}$}]{};

			\draw[thick,->] (7) -- (0);
			\draw[thick,->] (4) -- (11);
			\draw[thick,->] (11) -- (1.75,1);
                        \draw[thick,->] (1.15,1) -- (7);

			\node (12) at (2,-1)      [label={[label distance=-2pt]-90:$j_n$}]{};	
			\node (14) at (3,-1)     [label={[label distance=-6pt]-90:$j_{n-1}$}]{};
			\node (16) at (4,-1)    [label={[label distance=-6pt]-90:$j_{n-2}$}]{};

			\draw[thick,->] (6) -- (16);
			\draw[thick,->] (16) -- (14);		
			\draw[thick,->] (14) -- (12);
			\draw[thick,->] (12) -- (2);

			\node[outer sep=.1pt,inner sep=.5pt,minimum size=.01cm] at (1.1,0) {};	
			\node[outer sep=.1pt,inner sep=.5pt,minimum size=.01cm] at (1.2,0) {};
			\node[outer sep=.1pt,inner sep=.5pt,minimum size=.01cm] at (1.3,0) {};

			\node[outer sep=.1pt,inner sep=.5pt,minimum size=.01cm] at (2.6,0) {};
			\node[outer sep=.1pt,inner sep=.5pt,minimum size=.01cm] at (2.7,0) {};
			\node[outer sep=.1pt,inner sep=.5pt,minimum size=.01cm] at (2.8,0) {};

			\node[outer sep=.1pt,inner sep=.5pt,minimum size=.01cm] at (4.1,0) {};
			\node[outer sep=.1pt,inner sep=.5pt,minimum size=.01cm] at (4.2,0) {};
			\node[outer sep=.1pt,inner sep=.5pt,minimum size=.01cm] at (4.3,0) {};

			\node[outer sep=.1pt,inner sep=.5pt,minimum size=.01cm] at (1.35,1) {};
			\node[outer sep=.1pt,inner sep=.5pt,minimum size=.01cm] at (1.45,1) {};
			\node[outer sep=.1pt,inner sep=.5pt,minimum size=.01cm] at (1.55,1) {};
			
                      \end{tikzpicture}
                      \caption{Type 3}\label{fig:nonHamilton_type3}
                    \end{subfigure}

                    \bigskip
			\begin{subfigure}{\linewidth}
				\centering
				\begin{tikzpicture}[xscale=1.4,every node/.style={draw,fill=black,shape=circle,outer sep=1pt,inner sep=1pt,minimum size=.1cm}]
				
				\node (1)  at (0,0)    [label={[label distance=-2pt]-90:$1$}]{};
				\node (2)  at (.5,0)   [label={[label distance=-2pt]-90:$j_2$}]{};
				\node (3)  at (1.5,0)  [label={[label distance=-2pt]-90:$j_p$}]{};
				\node (4)  at (2,0)    [label={[label distance=-6pt]-90:$j_{k+1}$}]{};
				\node (5)  at (3,0)    [label={[label distance=-2pt]-90:$j_q$}]{};
				\node (6)  at (3.5,0)  [label={[label distance=-6pt]-90:$j_{q+1}$}]{};
				\node (7)  at (4.5,0)  [label={[label distance=-6pt]-90:$j_{r-3}$}]{};
				\node (8)  at (.25,1)  [label={[label distance=-2pt]90:$j_k$}]{};
				\node (9)  at (1.25,1) [label={[label distance=-6pt]90:$j_{p+1}$}]{};
				\node (10) at (3.25,1) [label={[label distance=-2pt]90:$j_r$}]{};
				\node (11) at (3.75,1) [label={[label distance=-6pt]90:$j_{r-1}$}]{};
				\node (12) at (4.25,1) [label={[label distance=-6pt]90:$j_{r-2}$}]{};
				\node (13) at (5,0)	[label={ [label distance= -1pt]-70:$j_{r+1}$}]{};
				\node (14) at (6,0)	[label={[label distance= 1pt] -90:$j_{n}$}]{};

				\draw[thick,->] (1) -- (2);
				\draw[thick,->] (2) -- (1,0);
				\draw[thick,->] (3) -- (4);
				\draw[thick,->] (4) -- (2.5,0);
				\draw[thick,->] (5) -- (6);
				\draw[thick,->] (6) -- (4,0);
				\draw[thick,->] (8) -- (1);
				\draw[thick,->] (9) -- (.75,1);
				\draw[thick,->] (7) -- (12);
				\draw[thick,->] (12) -- (11);
				\draw[thick,->] (11) -- (10);
				\draw[thick,->] (10) -- (5);
				\draw[thick,->] (3) -- (9);
				
				\node[outer sep=.1pt,inner sep=.5pt,minimum size=.01cm] at (1.1,0) {};
				\node[outer sep=.1pt,inner sep=.5pt,minimum size=.01cm] at (1.2,0) {};
				\node[outer sep=.1pt,inner sep=.5pt,minimum size=.01cm] at (1.3,0) {};
				\node[outer sep=.1pt,inner sep=.5pt,minimum size=.01cm] at (2.6,0) {};
				\node[outer sep=.1pt,inner sep=.5pt,minimum size=.01cm] at (2.7,0) {};
				\node[outer sep=.1pt,inner sep=.5pt,minimum size=.01cm] at (2.8,0) {};
				\node[outer sep=.1pt,inner sep=.5pt,minimum size=.01cm] at (4.1,0) {};
				\node[outer sep=.1pt,inner sep=.5pt,minimum size=.01cm] at (4.2,0) {};
				\node[outer sep=.1pt,inner sep=.5pt,minimum size=.01cm] at (4.3,0) {};
				\node[outer sep=.1pt,inner sep=.5pt,minimum size=.01cm] at (.65,1) {};
				\node[outer sep=.1pt,inner sep=.5pt,minimum size=.01cm] at (.45,1) {};
				\node[outer sep=.1pt,inner sep=.5pt,minimum size=.01cm] at (.55,1) {};
				\node[outer sep=.1pt,inner sep=.5pt,minimum size=.01cm] at (5.2,0) {};
				\node[outer sep=.1pt,inner sep=.5pt,minimum size=.01cm] at (5.4,0) {};
				\node[outer sep=.1pt,inner sep=.5pt,minimum size=.01cm] at (5.6,0) {};
				\node[outer sep=.1pt,inner sep=.5pt,minimum size=.01cm] at (5.8,0) {};
                              \end{tikzpicture}
                              \caption{Type 4}\label{fig:nonHamilton_type4}
                              \end{subfigure}
                            \end{minipage}
                          }
			\caption{Different types of non-Hamiltonian extreme points of the polytope $\hBeta(G)$}\label{fig:types}
\end{figure}

\begin{definition}\label{def:Ham_basis}
  Let $\vect{x}$ be an extreme point of the polytope $\hBeta(G)$ with corresponding support
  $\supportgraph(G,\vect x)$. A feasible basis that contains the arc set of $\supportgraph(G,\vect x)$ is
  called a \emph{Hamiltonian basis} if $\vect x$ is a Hamiltonian extreme point, and otherwise it is
  called a \emph{non-Hamiltonian basis}.
\end{definition}

\begin{remark}
  Analogously, the set of feasible bases of the polytope $\hBeta(G)$ can be partitioned into five
  types: Hamiltonian bases (namely, feasible bases of Type $0$) and non-Hamiltonian bases of types
  1--4, where a non-Hamiltonian basis is Type $i$ for $i=1,2,3,4$ if its corresponding extreme
  point is Type $i$.
\end{remark}

Since the supports of extreme points of types 1--3 have exactly $n+1$ elements, they are all
non-degenerate extreme points. This implies that each extreme point of these types has exactly one
corresponding feasible basis forming its support. However, as Hamiltonian and
non-Hamiltonian Type 4 extreme points have, respectively, exactly and at most $n$ positive
components, they are all degenerate extreme points and accordingly, they may possess several
corresponding feasible bases. Thus, the support $\supportgraph(G,\vect x)$ associated with a
Hamiltonian or non-Hamiltonian Type 4 extreme point $\vect{x}$ does not reveal complete
information about the feasible bases corresponding to $\vect{x}$.  For example,
Figure~\ref{fig:nonHam_EP}(\subref{fig:nonHamilton_Ext_2}) shows $\supportgraph(K_7,\vect x)$, where
$K_7$ is the complete graph on seven nodes
and $\vect{x}$ is the non-Hamiltonian extreme point of Type 4 with positive coordinates
$x_{12} = 1$, $x_{23} = \beta^3 (1+\beta^2) + \beta$, $x_{32} = \beta^2(1+\beta^2)$,
$x_{31} = \beta^6$. Clearly, in order to construct a feasible basis corresponding to this extreme
point, we should add four more appropriate arcs (not necessarily any four arbitrary arcs) to the
support given in Figure~\ref{fig:nonHam_EP}(\subref{fig:nonHamilton_Ext_2}). If we try the four arcs
$(4,5), (6,5), (7,6)$ and $(7,4)$, as in
Figure~\ref{fig:nonHam_EP}(\subref{fig:nonHamilton_basis_1}), this fails as it would induce linear
dependency. However, if we complete the basis with arcs $(3,4), (5,4), (6,7)$ and $(7,6)$, as in
Figure~\ref{fig:nonHam_EP}(\subref{fig:nonHamilton_basis_2}), this results in a feasible basis of
Type 4.

Thus, an important question raised here is which arcs can be added to the support of a
degenerate extreme point of the polytope $\mathcal{H}_\beta(G)$ to construct a corresponding
feasible basis? This question is addressed in Proposition~\ref{prop:Hamilton_basis} and
Theorem~\ref{thm:characterising_nonHamiltonian_bases} for Hamiltonian and non-Hamiltonian Type
4 extreme points, respectively.
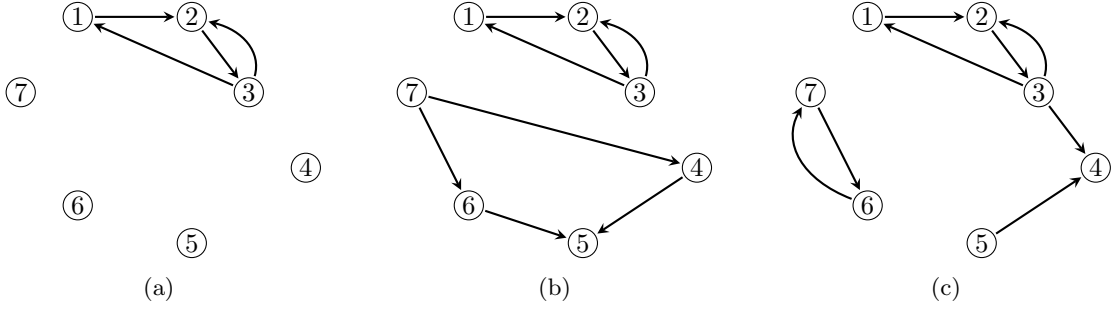
\begin{figure}
		\centering
		\begin{subfigure}{.3\textwidth}
			\centering
		\begin{tikzpicture}[xscale=1.5,every node/.style={draw,shape=circle,outer sep=1pt,inner sep=1pt,minimum size=.15cm}]
		\node (1) at (0,0) {1};
		\node (2) at (1,0) {2};
		\node (3) at (1.5,-1) {3};
		\node (4) at (2,-2) {4};
		\node (5) at (1,-3) {5};
		\node (6) at (0,-2.5) {6};
		\node (7) at (-.5,-1) {7};

		\draw[thick,->] (1) -- (2);
		\draw[thick,->] (2) -- (3);
		\draw[thick,->] (3) -- (1);
		\draw[thick,bend right=40,->] (3) to (2);
		\end{tikzpicture}
		\caption{\small }\label{fig:nonHamilton_Ext_2}
	\end{subfigure}
	\begin{subfigure}{.3\textwidth}
		\centering
		\begin{tikzpicture}[xscale=1.5,every node/.style={draw,shape=circle,outer sep=1pt,inner sep=1pt,minimum size=.15cm}]
				\node (1) at (0,0) {1};
				\node (2) at (1,0) {2};
				\node (3) at (1.5,-1) {3};
				\node (4) at (2,-2) {4};
				\node (5) at (1,-3) {5};
				\node (6) at (0,-2.5) {6};
				\node (7) at (-.5,-1) {7};

				\draw[thick,->] (1) -- (2);
				\draw[thick,->] (2) -- (3);
				\draw[thick,->] (3) -- (1);
				\draw[thick,bend right=40,->] (3) to (2);
				\draw[thick,->] (4) -- (5);
				\draw[thick,->] (6) -- (5);
				\draw[thick,->] (7) -- (6);
				\draw[thick,->] (7) -- (4);
		\end{tikzpicture}\caption{}\label{fig:nonHamilton_basis_1}
	\end{subfigure}
	\begin{subfigure}{.3\textwidth}
		\centering
			\begin{tikzpicture}[xscale=1.5,every node/.style={draw,shape=circle,outer sep=1pt,inner sep=1pt,minimum size=.15cm}]
			\node (1) at (0,0) {1};
			\node (2) at (1,0) {2};
			\node (3) at (1.5,-1) {3};
			\node (4) at (2,-2) {4};
			\node (5) at (1,-3) {5};
			\node (6) at (0,-2.5) {6};
			\node (7) at (-.5,-1) {7};

			\draw[thick,->] (1) -- (2);
			\draw[thick,->] (2) -- (3);
			\draw[thick,->] (3) -- (1);
			\draw[thick,bend right=40,->] (3) to (2);
			\draw[thick,->] (3) -- (4);
			\draw[thick,->] (5) -- (4);
			\draw[thick,bend left=40,->] (6) to (7);
			\draw[thick,->] (7) -- (6);
			\end{tikzpicture}
		\caption{}\label{fig:nonHamilton_basis_2}
	\end{subfigure}
	\caption{(a) Support of a non-Hamiltonian extreme point of Type 4 of $\hBeta(K_7)$, and two
        possible ways to add four arcs leading to a (b) linearly dependent set, and (c) a feasible basis.}\label{fig:nonHam_EP}
\end{figure}
\begin{proposition}\label{prop:Hamilton_basis}
  A set $B\subseteq \arcset$ of size $\lvert B\rvert=n+1$ is a Hamiltonian basis of the polytope
  $\hBeta(G)$ if and only if $B$ contains a Hamiltonian cycle.
\end{proposition}
\begin{proof}
  If $B$ is a Hamiltonian basis, it directly follows from Definitions~\ref{def:Ham_EP}
  and~\ref{def:Ham_basis} that it contains a Hamiltonian cycle. So, we just need to show that for
  any Hamiltonian cycle in $G$ with the arc set $C\subseteq E$ and any arc $(i,j)\in E\setminus C$,
  the set $B=C\cup \{(i,j)\}$ is a feasible basis for the polytope $\mathcal{H}_\beta(G)$. Without
  loss of generality, assume that
  \[C = \{(1,2),(2,3),\ldots,(n-1,n),(n,1)\}\]
  and fix any such $B$. In order to
  show that $B$ is a feasible basis, we need to show that conditions \textbf{B1} and \textbf{B2}
  hold. Let $A_B$ denote the $(n+1)\times (n+1)$-submatrix of $A$
  corresponding to $B$. If $A_B$ is invertible and $\vect x$ is the vector whose coordinates
  corresponding to $B$ are given by $A_B^{-1}\vect b$ while all other coordinates are zero, then it
   has been proved in \cite[Lemma 3.2]{eshragh2011hybrid} that
  $x_{k\,k+1}=\beta^{k-1}>0$ for $k\in\{1,2,\dotsc,n-1\}$, $x_{n1}=\beta^{n-1}>0$, and $x_{kl}=0$
  for all $(k,l)\in E\setminus C$. So the non-negativity condition \textbf{B2} is satisfied whenever
  the independence condition \textbf{B1} holds. Hence, we only need to prove that condition \textbf{B1}
  holds, that is the set $B$ is linearly independent.

  We prove this by contradiction. Let us assume that there exists a non-zero $(n+1)\times 1$
  vector $\vect x$ satisfying $A_B \vect x = \vect 0$. In this proof, all indices $0$ and $n+1$ that appear in arc-indicating subscripts are equivalent to $n$ and $1$, respectively; for example,
    $x_{n\,n+1} = x_{n1}$ and $x_{01} = x_{n1}$. Following
  constraints~(\ref{eq:flow_conservation_1})--(\ref{eq:flow_injection}), replacing the right-hand
  sides of~(\ref{eq:flow_conservation_1}) and~(\ref{eq:flow_injection}) by 0, this set of linear
  equations is represented as follows:
  \begin{align}
    \medskip x_{k\,k+1} - \beta x_{k-1\,k} & = 0 && \text{for }k\in V\setminus \{i,j\}, \label{eq:flow_con_k_prop1}\\
    \medskip x_{i\,i+1} + x_{i\,j} - \beta x_{i-1\,i} & = 0, \label{eq:flow_con_i_prop1}\\
    \medskip x_{j\,j+1} - \beta \left(x_{j-1\,j} + x_{i\,j}\right) & = 0, \label{eq:flow_con_j_prop1}\\
    \medskip x_{12} + \delta_{i\,1}x_{i\,j} & = 0, \label{eq:flow_injection_prop1}
  \end{align}
  where $\delta_{i\,1}$ is the Kronecker delta, which is equal to one if $i=1$, otherwise zero. If
  $x_{ij}=0$ then Equation~(\ref{eq:flow_injection_prop1}) implies that $x_{12}=0$, and then $x_{k\,k+1}=0$ for
  all $k\in\{1,2,\dots,n\}$ by induction on $k$,
  using~(\ref{eq:flow_con_k_prop1}),~(\ref{eq:flow_con_i_prop1}) and~(\ref{eq:flow_con_j_prop1}). Hence,
  our assumption $\vect x\neq\vect 0$ implies that $x_{i\,j} \neq 0$.  Therefore, without loss of
  generality, we assume $x_{i\,j}=1$. As illustrated in Figure~\ref{fig:cases}, there are four
  possible cases for the arc $(i,j)$ discussed below. We show that in all four cases the non-zero 
  assumption on $\vect x$ is violated, implying that the set $B$ is linearly
  independent.
  \tikzset{every node/.style={ellipse,inner sep=1pt,outer
      sep=2pt,draw,minimum width=1.2cm,minimum height=.45cm}}
	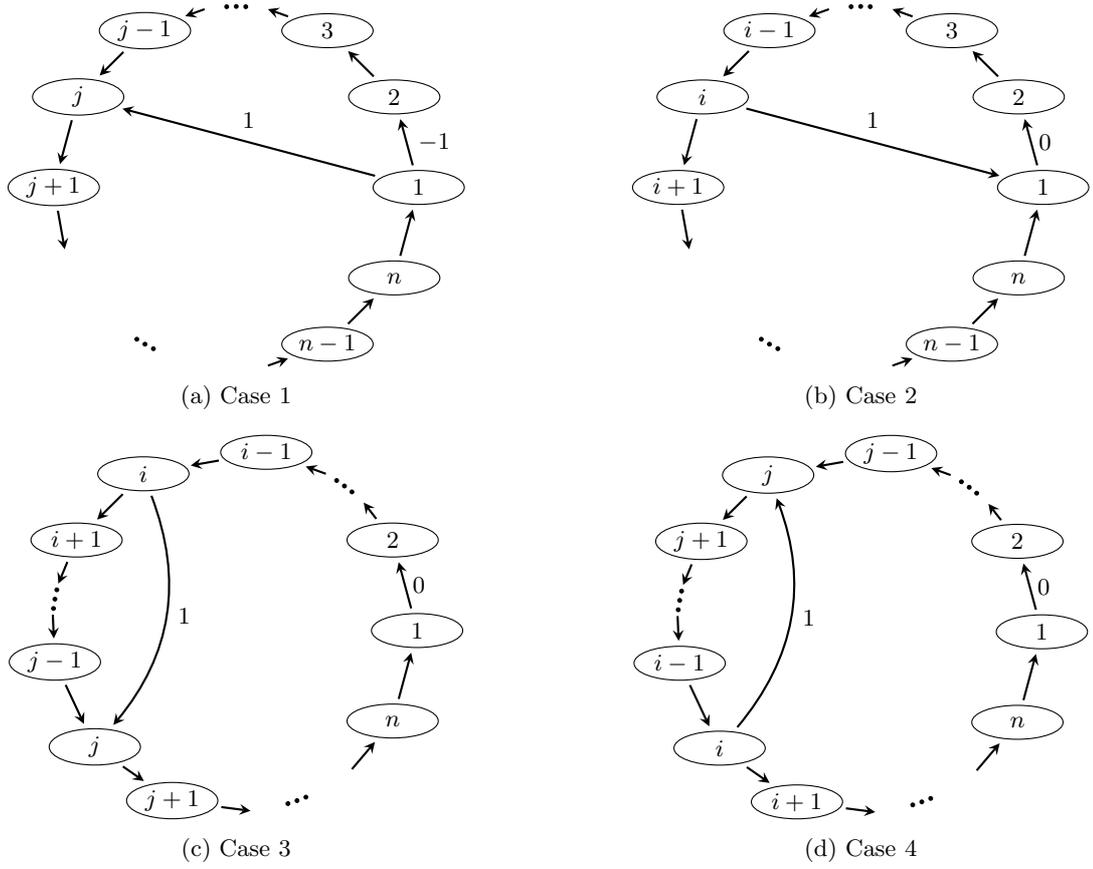
\begin{figure}
		{\small
			\centering
			\begin{subfigure}[b]{0.49\textwidth}
				\begin{center}
					\begin{tikzpicture}[scale=1.2]
					\node (v1) at (0:2) {{\small $1$}};
					\node (v2) at (30:2) {{\small $2$}};
					\node (v3) at (60:2) {{\small $3$}};
					\node (vn1) at (-60:2) {{\small $n-1$}};
					\node (vn) at (-30:2) {{\small $n$}};
					\node (vj1) at (120:2) {{\small $j-1$}};
					\node (vj) at (150:2) {{\small $j$}};
					\node (vj2) at (180:2) {{\small $j+1$}};
					\draw[thick,->] (v1) to node[draw=none,fill=none,above] {$1$} (vj);
					\node[circle,inner sep=0pt,minimum size=.5mm,fill=black,draw] at (87:2) {};
					\node[circle,inner sep=0pt,minimum size=.5mm,fill=black,draw] at (90:2) {};
					\node[circle,inner sep=0pt,minimum size=.5mm,fill=black,draw] at (93:2) {};          
					\draw[thick,->] (v1) to node[draw=none,fill=none,right,minimum size=0pt,inner sep=0pt] {{\small $-1$}} (v2);
					\draw[thick,->] (v2) -- (v3);
					\draw[thick,->] (vj1) -- (vj);
					\draw[thick,->] (vj) -- (vj2);
					\draw[thick,->] (vn1) -- (vn);
					\draw[thick,->] (vn) -- (v1);
					\draw[thick,->] (v3) -- (80:2);
					\draw[thick,->] (100:2) -- (vj1);
					\draw[thick,->] (vj2) -- (200:2);
					\draw[thick,->] (-80:2) -- (vn1);
					\node[circle,inner sep=0pt,minimum size=.5mm,fill=black,draw] at (237:2) {};
					\node[circle,inner sep=0pt,minimum size=.5mm,fill=black,draw] at (240:2) {};
					\node[circle,inner sep=0pt,minimum size=.5mm,fill=black,draw] at (243:2) {};
					\end{tikzpicture}
					\caption{Case 1}
				\end{center}
			\end{subfigure}%
			\begin{subfigure}[b]{0.49\textwidth}
				\centering
				\begin{tikzpicture}[scale=1.2]
				\node (v1) at (0:2) {{\small $1$}};
				\node (v2) at (30:2) {{\small $2$}};
				\node (v3) at (60:2) {{\small $3$}};
				\node (vn1) at (-60:2) {{\small $n-1$}};
				\node (vn) at (-30:2) {{\small $n$}};
				\node (vj1) at (120:2) {{\small $i-1$}};
				\node (vj) at (150:2) {{\small $i$}};
				\node (vj2) at (180:2) {{\small $i+1$}};
				\draw[thick,<-] (v1) to node[draw=none,fill=none,above] {$1$} (vj);           
				\node[circle,inner sep=0pt,minimum size=.5mm,fill=black,draw] at (87:2) {};
				\node[circle,inner sep=0pt,minimum size=.5mm,fill=black,draw] at (90:2) {};
				\node[circle,inner sep=0pt,minimum size=.5mm,fill=black,draw] at (93:2) {};          
				\draw[thick,->] (v1) to node[draw=none,fill=none,right,minimum size=0pt,inner sep=0pt] {{\small $0$}} (v2);
				\draw[thick,->] (v2) -- (v3);
				\draw[thick,->] (vj1) -- (vj);
				\draw[thick,->] (vj) -- (vj2);
				\draw[thick,->] (vn1) -- (vn);
				\draw[thick,->] (vn) -- (v1);
				\draw[thick,->] (v3) -- (80:2);
				\draw[thick,->] (100:2) -- (vj1);
				\draw[thick,->] (vj2) -- (200:2);
				\draw[thick,->] (-80:2) -- (vn1);
				\node[circle,inner sep=0pt,minimum size=.5mm,fill=black,draw] at (237:2) {};
				\node[circle,inner sep=0pt,minimum size=.5mm,fill=black,draw] at (240:2) {};
				\node[circle,inner sep=0pt,minimum size=.5mm,fill=black,draw] at (243:2) {};
				\end{tikzpicture}
				\caption{Case 2}
			\end{subfigure}
			
			\bigskip
			
			\begin{subfigure}[b]{0.49\textwidth}
				\centering
				\begin{tikzpicture}[scale=1.2]
				\node (v1) at (0:2) {{\small $1$}};
				\node (v2) at (30:2) {{\small $2$}};          
				\node[circle,inner sep=0pt,minimum size=.5mm,fill=black,draw] at (50:2) {};
				\node[circle,inner sep=0pt,minimum size=.5mm,fill=black,draw] at (53:2) {};
				\node[circle,inner sep=0pt,minimum size=.5mm,fill=black,draw] at (56:2) {};
				\node (vn) at (-30:2) {{\small $n$}};
				\node (vi1) at (80:2) {{\small $i-1$}};
				\node (vi) at (120:2) {{\small $i$}};
				\node (vi2) at (150:2) {{\small $i+1$}};
				\node[circle,inner sep=0pt,minimum size=.5mm,fill=black,draw] at (167:2) {};
				\node[circle,inner sep=0pt,minimum size=.5mm,fill=black,draw] at (170:2) {};
				\node[circle,inner sep=0pt,minimum size=.5mm,fill=black,draw] at (173:2) {};
				\node (vj1) at (190:2) {{\small $j-1$}};
				\node (vj) at (220:2) {{\small $j$}};
				\node (vj2) at (250:2) {{\small $j+1$}};
				\node[circle,inner sep=0pt,minimum size=.5mm,fill=black,draw] at (287:2) {};
				\node[circle,inner sep=0pt,minimum size=.5mm,fill=black,draw] at (290:2) {};
				\node[circle,inner sep=0pt,minimum size=.5mm,fill=black,draw] at (293:2) {};
				\draw[thick,bend left=30,->] (vi) to node[draw=none,fill=none,right,minimum size=0] {$1$}
				(vj);
				\draw[thick,->] (v1) to node[draw=none,fill=none,right,minimum size=0pt,inner sep=0pt] {{\small $0$}} (v2);
				\draw[thick,->] (v2) -- (45:2);
				\draw[thick,->] (60:2) -- (vi1);
				\draw[thick,->] (vi1) -- (vi);
				\draw[thick,->] (vi) -- (vi2);
				\draw[thick,->] (vi2) -- (165:2);
				\draw[thick,->] (175:2) -- (vj1);
				\draw[thick,->] (vj1) -- (vj);
				\draw[thick,->] (vj) -- (vj2);
				\draw[thick,->] (vj2) -- (275:2);
				\draw[thick,->] (-50:2) -- (vn);
				\draw[thick,->] (vn) -- (v1);
				\end{tikzpicture}
				\caption{Case 3}
			\end{subfigure}%
			\begin{subfigure}[b]{0.49\textwidth}
				\centering
				\begin{tikzpicture}[scale=1.2]
				\node (v1) at (0:2) {{\small $1$}};
				\node (v2) at (30:2) {{\small $2$}};          
				\node[circle,inner sep=0pt,minimum size=.5mm,fill=black,draw] at (50:2) {};
				\node[circle,inner sep=0pt,minimum size=.5mm,fill=black,draw] at (53:2) {};
				\node[circle,inner sep=0pt,minimum size=.5mm,fill=black,draw] at (56:2) {};
				\node (vn) at (-30:2) {{\small $n$}};
				\node (vi1) at (80:2) {{\small $j-1$}};
				\node (vi) at (120:2) {{\small $j$}};
				\node (vi2) at (150:2) {{\small $j+1$}};
				\node[circle,inner sep=0pt,minimum size=.5mm,fill=black,draw] at (167:2) {};
				\node[circle,inner sep=0pt,minimum size=.5mm,fill=black,draw] at (170:2) {};
				\node[circle,inner sep=0pt,minimum size=.5mm,fill=black,draw] at (173:2) {};
				\node (vj1) at (190:2) {{\small $i-1$}};
				\node (vj) at (220:2) {{\small $i$}};
				\node (vj2) at (250:2) {{\small $i+1$}};
				\node[circle,inner sep=0pt,minimum size=.5mm,fill=black,draw] at (287:2) {};
				\node[circle,inner sep=0pt,minimum size=.5mm,fill=black,draw] at (290:2) {};
				\node[circle,inner sep=0pt,minimum size=.5mm,fill=black,draw] at (293:2) {};
				\draw[thick,bend left=30,<-] (vi) to node[draw=none,fill=none,right,minimum size=0] {$1$}
				(vj);
				\draw[thick,->] (v1) to node[draw=none,fill=none,right,minimum size=0pt,inner sep=0pt] {{\small $0$}} (v2);
				\draw[thick,->] (v2) -- (45:2);
				\draw[thick,->] (60:2) -- (vi1);
				\draw[thick,->] (vi1) -- (vi);
				\draw[thick,->] (vi) -- (vi2);
				\draw[thick,->] (vi2) -- (165:2);
				\draw[thick,->] (175:2) -- (vj1);
				\draw[thick,->] (vj1) -- (vj);
				\draw[thick,->] (vj) -- (vj2);
				\draw[thick,->] (vj2) -- (275:2);
				\draw[thick,->] (-50:2) -- (vn);
				\draw[thick,->] (vn) -- (v1);
				\end{tikzpicture}
				\caption{Case 4}
			\end{subfigure}    
			\caption{Illustration for the four cases discussed in the proof of Proposition~\ref{prop:Hamilton_basis}.}
			\label{fig:cases}
		}
\end{figure}
\begin{description}
\item[Case 1] $i=1$, $3\leq j\leq n$. From Equation~(\ref{eq:flow_injection_prop1}) we obtain $x_{12}=-1$, and then
  by Equation~(\ref{eq:flow_con_k_prop1}) and induction on $k$ we have $x_{k\,k+1}=-\beta^{k-1}$ for $k =
  1,\dotsc,j-1$. So, Equation~(\ref{eq:flow_con_j_prop1}) simplifies to
  \[x_{j\,j+1}-\beta\left(-\beta^{j-2}+1\right)=0,\]
  implying that $x_{j\,j+1}=\beta-\beta^{j-1}$. Substituting this value into Equation~(\ref{eq:flow_con_k_prop1}),
  and proceeding by induction on $k$, we have 
  \[x_{k\,k+1}=\beta^{k-j+1}-\beta^{k-1}\quad\text{for } k=j,j+1,\dotsc,n.\]
  Hence, the left-hand side of Equation~(\ref{eq:flow_con_i_prop1}) is 
  \[x_{12}+x_{1j}-\beta x_{n1}=-1+1-\beta\left(\beta^{n-j+1}-\beta^{n-1}\right)< 0.\]
  In particular, $A_B\vect x\neq\vect 0$, which is the required contradiction.
\item[Case 2] $j=1$, $2\leq i\leq n-1$. From Equation~(\ref{eq:flow_injection_prop1}), we obtain
  $x_{12}=0$. As in Case 1, we have 
  \[x_{k\,k+1}=
    \begin{cases}
      0 &\text{for }k=1,2,\dots,i-1,\\
      -\beta^{k-i}&\text{for }k=i,i+1,\dots,n.
    \end{cases}
\]
Hence, the left-hand side of Equation~(\ref{eq:flow_con_j_prop1}) is 
  \[x_{12}-\beta \left(x_{n1}+x_{i1}\right)=0-\beta\left(-\beta^{n-i}+1\right)< 0.\]
  In particular, $A_B\vect x\neq\vect 0$, which is the required contradiction.
\item[Case 3] $2\leq i<j-1\leq n-1$. As in Case 2, we have
  \[x_{k\,k+1}=
    \begin{cases}
      0 & \text{for }k=1,2,\dots,i-1,\\
      -\beta^{k-i} &\text{for }k=i,i+1,\dots,j-1,\\
      \beta^{k-j+1}-\beta^{k-i}&\text{for }k=j,j+1,\dots,n.
    \end{cases}
\]
Hence, the left-hand side of Equation~(\ref{eq:flow_con_k_prop1}) for $k=1$ is 
  \[x_{12}-\beta x_{n1}=0-\beta\left(\beta^{n-j+1}-\beta^{n-i}\right)< 0.\]
In particular, $A_B\vect x\neq\vect 0$, which is the required contradiction.
\item[Case 4] $2\leq j<i\leq n$. As in Case 2, we have
  \[x_{k\,k+1}=
    \begin{cases}
      0 & \text{for }k=1,2,\dots,j-1,\\
      \beta^{k-j+1} &\text{for }k=j,j+1,\dots,i-1,\\
      \beta^{k-j+1}-\beta^{k-i}&\text{for }k=i,i+1,\dots,n.
    \end{cases}
\]
Hence, the left-hand side of Equation~(\ref{eq:flow_con_k_prop1}) for $k=1$ is
  \[x_{12}-\beta x_{n1}=0-\beta\left(\beta^{n-j+1}-\beta^{n-i}\right)> 0.\]
In particular, $A_B\vect x\neq\vect 0$, which is the required contradiction.\qedhere
\end{description}
\end{proof}

\begin{corollary}\label{lem:adj_HExt}
  Every Hamiltonian extreme point of $\hBeta(G)$ corresponds to $\lvert E\rvert-n$ Hamiltonian bases,
  and any two of these bases are connected by a single variable exchange.
\end{corollary}
\begin{proof}
  Proposition~\ref{prop:Hamilton_basis} implies that a Hamiltonian basis of $\hBeta(G)$ is
  obtained by adding an arbitrary arc to the arc set of any Hamiltonian cycle in $G$.
\end{proof}

For example, Figure~\ref{fig:Hamiltonian basis} illustrates three different Hamiltonian bases of the
polytope $\mathcal{H}_\beta(G)$ for the input graph $G$  given in
Figure~\ref{fig:Hamilton_nonHamilton_graph}(\subref{fig:Hamilton_graph}). All those three
Hamiltonian bases correspond to the Hamiltonian extreme point with the support displayed in Figure~\ref{fig:supports}(\subref{fig:HamiltonSupport}).

\begin{figure}
	\begin{subfigure}{.3\textwidth}
		\centering
		\begin{tikzpicture}[xscale=1.5,every node/.style={draw,shape=circle,outer sep=1pt,inner sep=1pt,minimum size=.15cm}]
		\node (1) at (0,0) {1};
		\node (2) at (1,0) {2};
		\node (3) at (1.5,-1) {3};
		\node (4) at (.5,-1) {4};
		\node (5) at (-.5,-1) {5};

		\draw[thick,->] (1) -- (2);
		\draw[thick,->] (2) -- (3);
		\draw[thick,->] (3) -- (4);
		\draw[thick,->] (4) -- (5);
		\draw[thick,->] (5) -- (1);
		\draw[thick,bend right=60,->,dashed] (2) to (1);
		
		\end{tikzpicture}
		\caption{}
	\end{subfigure}
	\hspace*{\fill}
	\begin{subfigure}{.3\textwidth}
		\centering
		\begin{tikzpicture}[xscale=1.5,every node/.style={draw,shape=circle,outer sep=1pt,inner sep=1pt,minimum size=.15cm}]
		\node (1) at (0,0) {1};
		\node (2) at (1,0) {2};
		\node (3) at (1.5,-1) {3};
		\node (4) at (.5,-1) {4};
		\node (5) at (-.5,-1) {5};

		\draw[thick,->] (1) -- (2);
		\draw[thick,->] (2) -- (3);
		\draw[thick,->] (3) -- (4);
		\draw[thick,->] (4) -- (5);
		\draw[thick,->] (5) -- (1);
		\draw[thick,->,dashed] (1) -- (4);
		\end{tikzpicture}
		\caption{}
	\end{subfigure}
		\hspace*{\fill}
		\begin{subfigure}{.3\textwidth}
			\centering
			\begin{tikzpicture}[xscale=1.5,every node/.style={draw,shape=circle,outer sep=1pt,inner sep=1pt,minimum size=.15cm}]
			\node (1) at (0,0) {1};
			\node (2) at (1,0) {2};
			\node (3) at (1.5,-1) {3};
			\node (4) at (.5,-1) {4};
			\node (5) at (-.5,-1) {5};

			\draw[thick,->] (1) -- (2);
			\draw[thick,->] (2) -- (3);
			\draw[thick,->] (3) -- (4);
			\draw[thick,->] (4) -- (5);
			\draw[thick,->] (5) -- (1);
			\draw[thick,bend right=60,->,dashed] (4) to (3);
			\end{tikzpicture}
			\caption{}
		\end{subfigure}
	\caption{Three Hamiltonian bases for the graph in Figure~\ref{fig:Hamilton_nonHamilton_graph}(\subref{fig:Hamilton_graph})}\label{fig:Hamiltonian basis}
\end{figure}
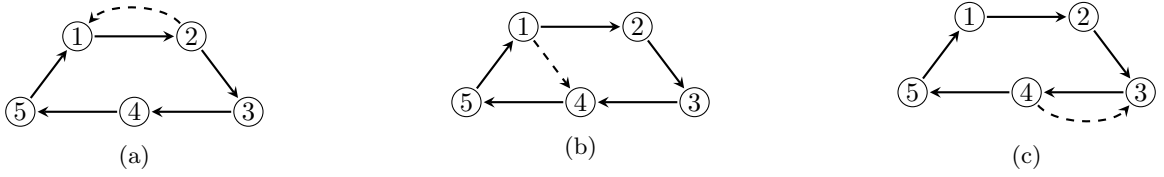

In the remainder of this section, we characterize the non-Hamiltonian bases of the polytope
$\hBeta(G)$. In order to formulate the necessary and sufficient conditions on an arc set $B\subseteq \arcset$ to be a non-Hamiltonian basis, we need to introduce the
following concept, and develop some preliminary results presented in
lemmas~\ref{lem:cycles}--\ref{lem:zero_arcs}.

\begin{definition}\label{def:balanced_cycle}
  Let $P=\{\{v_0,v_1\},\{v_1,v_2\},\dotsc,\{v_{k-1},v_k\}\}$, where $v_i \in \nodeset$ for $i=0,\dots,k$, be an undirected path (or cycle if
  $v_k=v_0$) in the graph $G$. An \emph{oriented} path (or cycle) of $P$ is denoted by $\overrightarrow{P}$ and defined by
\[\overrightarrow{P}=\{(u_1,w_1),(u_2,w_2),\dotsc,(u_k,w_k)\},\]
where $(u_i,w_i)\in\{(v_{i-1},v_i),(v_i,v_{i-1})\}$ for $i=1,\dotsc,k$. We call the arc $(u_i,w_i)$
a \emph{forward arc} if $(u_i,w_i)=(v_{i-1},v_i)$ and a \emph{backward arc} if
$(u_i,w_i)=(v_{i},v_{i-1})$. For the special cases where $\overrightarrow{P}$ is a path of length
one or a cycle of length two, that is $\overrightarrow{P}=\{(u_1,w_1)\}$ or
$\{(u_1,w_1),(w_1,u_1)\}$, we use the convention that all arcs are forward arcs. The \emph{defect}
of an oriented path (or cycle) $\overrightarrow{P}$ is denoted by $\Delta(\overrightarrow{P})$ and
defined to be the number of forward arcs minus the number of backward arcs in $\overrightarrow{P}$. We say that
$\overrightarrow{P}$ is \emph{balanced} if its defect is zero.
 \end{definition}
 \begin{figure}
 	\centering
 	\begin{subfigure}{.4 \textwidth}
 		\centering
 		\begin{tikzpicture}[xscale=1.5,every node/.style={draw,shape=circle,outer sep=1pt,inner sep=1pt,minimum size=.15cm}]
 		\node (1) at (0,0) {$1$};
 		\node (2) at (1,0) {$2$};
 		\node (3) at (1,-1){$3$};
 		\node (4) at (0,-1) {$4$};

 		\draw[thick,->] (1) -- (2);
 		\draw[thick,->,dashed]  (3) -- (2);
 		\draw[thick,->] (3) -- (4);
 		\draw[thick,->,dashed]  (1) -- (4);
 		\end{tikzpicture}\caption{A balanced oriented cycle}\label{fig:balanced_cycle}
 	\end{subfigure}
 	\hspace*{\fill}
 	\begin{subfigure}{.5 \textwidth}
 		\centering
 		\begin{tikzpicture}[xscale=1.5,every node/.style={draw,shape=circle,outer sep=1pt,inner sep=1pt,minimum size=.15cm}]
 		\node (1) at (0,0) {$1$};
 		\node (2) at (1,0) {$2$};
 		\node (3) at (2,0) {$3$};
 		\node (4) at (2.5,-1) {$4$};
 		\node (5) at (1,-1)	{$5$};
 		\node (6) at (0,-1) {$6$};
 		\node (7) at (-1,0) {$7$};
 		
 		\draw[thick,->,dashed ] (2) -- (1);
 		\draw[thick,->,dashed ] (3) -- (2);
 		\draw[thick,->] (3) -- (4);
 		\draw[thick,->] (4) -- (5);
 		\draw[thick,->,dashed ] (6) -- (5);
 		\draw[thick,->] (6) -- (7);		
 		\end{tikzpicture}\caption{A balanced oriented path}\label{fig:balanced_path}
 	\end{subfigure}\caption{Balanced oriented paths}\label{fig:balanced}
 \end{figure}
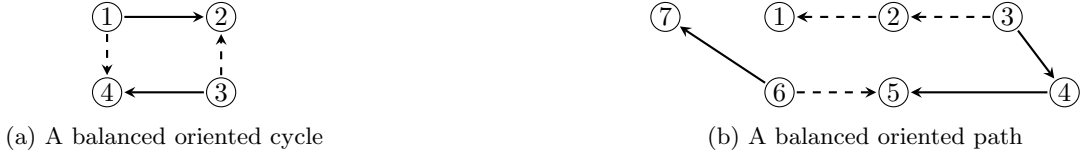

 Figures~\ref{fig:balanced}(\subref{fig:balanced_cycle})
 and~\ref{fig:balanced}(\subref{fig:balanced_path}) depict a balanced oriented cycle and path,
 respectively. While the solid arcs indicate forward arcs, the dashed arcs are backward arcs.
\begin{remark}
  Any balanced oriented path (or cycle) contains an even number of arcs, where half of them are
  forward arcs and half are backward arcs. Obviously, a Hamiltonian cycle, short cycle,
  noose path and noose cycle are non-balanced and all their arcs are forward arcs.
\end{remark}

\begin{lemma}\label{lem:cycles}
  Let $\vect x= (x_{ij})_{\lvert \arcset \rvert \times 1}$, where $(i,j)\in \arcset$, be a
  real-valued vector whose support is an oriented cycle $\overrightarrow C\subseteq E$ that does not contain node $1$. Then
  $A\vect x=\vect 0$ if and only if the following three conditions are satisfied:
	\begin{enumerate}[(i)]
		\item $x_{ij}=-x_{ik}$ if node $i$ has out-degree 2 in $\overrightarrow C$ and $(i,j),(i,k)\in\overrightarrow C$;
		\item $x_{ji}=-x_{ki}$ if node $i$ has in-degree 2 in $\overrightarrow C$ and $(j,i),(k,i)\in\overrightarrow C$;
		\item $x_{ij}=\beta x_{ki}$ if node $i$ has in- and out-degree 1 in
                  $\overrightarrow C$ and $(i,j),(k,i)\in\overrightarrow C$.
	\end{enumerate}
\end{lemma}
\begin{proof}
	This is an immediate consequence of constraint~(\ref{eq:flow_conservation_i}).
\end{proof}
\begin{lemma}\label{lem:balanced_cycles}
  Let $\overrightarrow C\subseteq E$ be an oriented cycle not containing node $1$. The set of
  columns of $A$ corresponding to the arcs in $\overrightarrow C$ is linearly dependent if and only if
  $\overrightarrow C$ is balanced.
\end{lemma}
\begin{proof}
  We first prove the `only if' statement by showing that if the set of columns of $A$ corresponding to $\overrightarrow{C}$ is linearly dependent, then $\overrightarrow{C}$ is a balanced oriented
  cycle. As all results in this proof are derived in `if and only if' condition, the converse, that
  is the `if' statement, can easily be shown by following the `only if' proof backward. Let
  $A_{\overrightarrow{C}}$ denote the $(n+1)\times k$ submatrix constructed by choosing the columns of $A$ corresponding to arcs in the oriented cycle $\overrightarrow{C}=\{e_1,e_2,\ldots,e_k\}$, where
  $e_l=(i_l,j_l)\in E$ for $l=1,\ldots,k$ and $k<n$. Define $\overrightarrow{P}_l=\{e_1,e_2,\ldots,e_l\}$ as
  the oriented subpath of the oriented cycle $\overrightarrow{C}$ for $l=1,\ldots,k$. Let $i^*$ be the common
  node of arcs $e_1$ and $e_k$, and $\sigma_l$ be the number of nodes in $\overrightarrow{P}_l$ having in- or
  out-degree 2. For $l=k$, we do not consider $i^*$ in the calculation of $\sigma_k$. From our
  linear dependency assumption, we know that there exists a non-zero $k\times 1$ vector $\vect{x}$
  such that $A_{\overrightarrow{C}}\vect{x} = \vect{0}$. Lemma 1 prescribes that $x_{i_l j_l} \neq 0$ for all
  $(i_l ,j_l)\in\overrightarrow{C}$. So, without loss of generality, we can re-scale $\vect{x}$ and define
	\begin{align}
          \medskip x_{i_1 j_1} & = \begin{cases}
            1\ \ \ \ \ \mbox{if } i_1=i^*\ \mbox{(i.e., $e_1$ is a forward arc)}, \\
            \beta^{-1}\ \ \mbox{if } j_1=i^*\ \mbox{(i.e., $e_1$ is a backward arc)}.
		\end{cases}\label{eq:x_i1j1}
	\end{align}
	By applying Lemma 1 and induction on $l$, we can show the following equality for the
        component of vector $\vect{x}$ corresponding to the last arc in $\overrightarrow{P}_l$, for
        $l=1,\dots,k$: 
	\begin{align}
		\medskip x_{i_l j_l} & = \begin{cases}
			(-1)^{\sigma_l} \beta^{\Delta(\overrightarrow{P}_l)-1}\ \ \mbox{if } e_l=(i_l,j_l) \mbox{ is a forward arc}, \\
			(-1)^{\sigma_l} \beta^{\Delta(\overrightarrow{P}_l)}\ \ \ \ \ \mbox{if } e_l=(i_l,j_l) \mbox{ is a backward arc}.
		\end{cases}\label{eq:x_iljl}
	\end{align}	
	As the number of nodes with in- or out-degree 2 in the oriented cycle $\overrightarrow{C}$ is even,
        $i^*$ has in- or out-degree 2, if and only if $\sigma_k$ is odd. For node $i^*$ in
        $\overrightarrow{P}_k$ (which is identical to $\overrightarrow{C}$), we have the following four possible cases:
	\begin{description}
        \item[Case 1] $i_1=i_k=i^*$. In this case, $e_1=(i^*,j_1)$ is a forward arc and
          $e_k=(i^*,j_k)$ is a backward arc. As the out-degree of node $i^*$ is 2, $\sigma_k$ is
          an odd number. Consequently, from equations~\eqref{eq:x_i1j1} and~\eqref{eq:x_iljl} we have $x_{i^* j_1}=1$ and
          $x_{i^* j_k} = -\beta^{\Delta(\overrightarrow{C})}$, respectively. However, Lemma~\ref{lem:cycles} prescribes that
          $x_{i^* j_1} = -x_{i^* j_k}$. Hence, in this case, we must have $\Delta(\overrightarrow{C})=0$.

        \item[Case 2] $i_1=j_k=i^*$. In this case, both $e_1=(i^*,j_1)$ and $e_k=(i_k,i^*)$ are
          forward arcs. As the in- and out-degree of node $i^*$ are 1, $\sigma_k$ is an even
          number. Consequently, from equations~\eqref{eq:x_i1j1} and~\eqref{eq:x_iljl} we have $x_{i^* j_1}=1$ and
          $x_{i_k i^*} = \beta^{\Delta(\overrightarrow{C})-1}$, respectively. However, Lemma~\ref{lem:cycles} prescribes that
          $x_{i^* j_1} = \beta x_{i_k i^*}$. Hence, in this case, we must have $\Delta(\overrightarrow{C})=0$.

        \item[Case 3] $j_1=i_k=i^*$. In this case, both $e_1=(i_1,i^*)$ and $e_k=(i^*,j_k)$ are
          backward arcs. As the in- and out-degree of node $i^*$ is 1, $\sigma_k$ is an even
          number. Consequently, from equations~\eqref{eq:x_i1j1} and~\eqref{eq:x_iljl} we have $x_{i_1 i^*}=\beta^{-1}$ and
          $x_{i^* j_k} = \beta^{\Delta(\overrightarrow{C})}$, respectively. However, Lemma 1 prescribes that
          $x_{i^* j_k} = \beta x_{i_1 i^*}$. Hence, in this case, we must have $\Delta(\overrightarrow{C})=0$.

        \item[Case 4] $j_1=j_k=i^*$. In this case, $e_1=(i_1, i^*)$ is a backward arc and
          $e_k=(i_k, i^*)$ is a forward arc. As the in-degree of node $i^*$ is 2, $\sigma_k$ is an
          odd number. Consequently, from equations~\eqref{eq:x_i1j1} and~\eqref{eq:x_iljl} we have $x_{i_1 i^*}=\beta^{-1}$ and
          $x_{i_k i^*} = -\beta^{\Delta(\overrightarrow{C})-1}$, respectively. However, Lemma~\ref{lem:cycles} prescribes that
          $x_{i_1 i^*} = -x_{i_k i^*}$. Hence, in this case, we must have $\Delta(\overrightarrow{C})=0$.
	\end{description}
	Thus, $\overrightarrow{C}$ is a balanced oriented cycle.
\end{proof}
\begin{lemma}\label{lem:good_augmented_trees}
  If $H=(W,B)$, where $W\subseteq V\setminus\{1\}$ and $B \subseteq \arcset$, is a connected subgraph of $G$ with
  $\lvert B\rvert=\lvert W\rvert$, then the set of columns of $A$ corresponding to arcs in $B$ is
  linearly dependent if and only if the unique oriented cycle in $B$ is balanced.
\end{lemma}
\begin{proof}
  If the oriented cycle in $B$ is balanced then $B$ is linearly dependent by
  Lemma~\ref{lem:balanced_cycles}. For the converse, suppose that $B$ is linearly dependent and let
  $B'\subseteq B$ be a minimal dependent subset. This implies that there exists a nonzero vector
  $\vect x$ with $A_{B'}\vect x=\vect 0$, where $A_{B'}$ is a submatrix constructed by the columns
  of $A$ corresponding to $B'$, such that for every arc $(i,j)\in B'$, the
  component $x_{ij}\neq 0$. If the subgraph $H'=(W,B')$ had a node $i$ of degree 1, then
  constraint~(\ref{eq:flow_conservation_i}) for this node would imply that $x_{ij}=0$ (if
  $(i,j)\in B'$) or $x_{ji}=0$ (if $(j,i)\in B'$). Since this contradicts the minimality of $B'$, we
  conclude that $B'$ is exactly the arc set of the oriented cycle in
  $B$. Lemma~\ref{lem:balanced_cycles} prescribes that $B'$ is balanced.
\end{proof}

In Lemma~\ref{lem:zero_arcs} and Theorem~\ref{thm:characterising_nonHamiltonian_bases} we use the notation $H=(V,B)$, where $B\subseteq E$ with $\lvert B \rvert=n+1$, is a subgraph of the graph $G$ with connected components $H_1,H_2,\dotsc,H_m$. Moreover, $H_k=(V_k,B_k)$, where $V_k$ and $B_k$ denote the set of nodes and arcs comprising the connected component $H_k$ for $k=1,\dotsc,m$, respectively. Without loss of generality, we assume that $1\in V_1$. Let $\rho(H_1)$ denote a subgraph of $H_1$ constructed by repeatedly removing all nodes in $H_1$ with degree equal to 1. We use the notation $V_\rho$ and $B_\rho$ to denote the node set and arc set of the graph $\rho(H_1)$. It should be noted that the concept of constructing the subgraph $\rho(H_1)$ is analogous to the 2-core of the graph $H_1$ (see, for example \cite{bollobas1984evolution} or \cite{Luczak1991}).

\begin{lemma}\label{lem:zero_arcs}
  Consider the subgraph $H = (\nodeset,B)$ with connected components $H_1,\dots,H_m$. Let $A_B$ be a submatrix of $A$ corresponding to the arcs of
  $B$. If $A_B$ is an invertible matrix and $\vect{x} = A_B^{-1} \vect{b}$, then
  \begin{enumerate}[(i)]
  \item $\lvert B_k\rvert=
			\begin{cases}
			\lvert V_1 \rvert+1&\mbox{for } k=1,\\
			
			\lvert V_k\rvert  & \text{for }k = 2,\ldots,m;\\
			\end{cases}$\label{lem:condition_1}
  \item $x_{ij} = 0$ for each $(i,j) \in B\setminus B_\rho$.		
  \end{enumerate}
\end{lemma}	

\begin{proof}
  \begin{enumerate}[(i)]
  \item Since $A_B$ is an invertible $(\lvert \nodeset \rvert+1) \times \lvert B
    \rvert$ square matrix,  $ \rank(A_B)=\lvert \nodeset \rvert +1= \lvert B
    \rvert=
    n+1$. Moreover, it follows from
    constraints~\eqref{eq:flow_conservation_1}--\eqref{eq:flow_injection} that 
    $A_B$ has a block structure, with non-overlapping blocks
    $A_{B_1},\dots,A_{B_m}$ such that $A_{B_1}$ is a $(\lvert \nodeset_1 \rvert+1) \times \lvert B_1
    \rvert$ matrix corresponding to the nodes and arcs in $V_1$ and
    $B_1$, respectively, and $A_{B_k}$ is a $\lvert \nodeset_k \rvert \times \lvert B_k
    \rvert$ matrix corresponding to the nodes and arcs in $V_k$ and $B_k$, respectively, for $k
    =2,\dots,m$.  Since there are no arcs between distinct components $H_i$ and $H_j$ for $i,j \in
    \{1,\dots,m\}$ and $i\ne j$, and $ \rank(A_B)=\lvert
    \nodeset\rvert+1$, we should have $ \rank(A_{B_1})=\lvert \nodeset_1 \rvert +1
    $ and $ \rank(A_{B_k})=\lvert \nodeset_k \rvert$, for $k= 2,\dots,
    m$. Similarly, as $\rank(A_{B})\lvert
    B\rvert$, we should also have $\rank(A_{B_k})=A_{B_k}$ equal to $ \lvert B_k \rvert$, for $k=
    1,\dots, m$. Thus,
	\[
	\lvert B_k\rvert=
	\begin{cases}
	\lvert\nodeset_1\rvert+1&\mbox{for } k=1,\\	
	\lvert\nodeset_k\rvert  & \text{for }k = 2,\ldots,m.
	\end{cases}\]	
\item Let $B'\subseteq B$ be the support of $\vect{x}$. We define $\vect{x}^k$ as the restriction of $\vect{x}$ to the coordinates in $B_k$ for $k=2,\dotsc,m$. Constraint~(\ref{eq:flow_conservation_i}) implies that $A_{B_k}\vect{x}^k = \vect{0}$, and since  $A_{B_k}$ has full column rank, we obtain $\vect{x}^k = \vect{0}$ for $k=2,\ldots,m$. Hence, $B'\subseteq B_1$. Moreover, if the subgraph $H'=(V_1,B')$ has a node $i$ of degree 1, then constraint~(\ref{eq:flow_conservation_i}) for node $i$ implies that  $x_{ij}=0$ (if $(i,j)\in B'$) or $x_{ji}=0$
(if $(j,i)\in B'$), which contradicts the definition of $B'$. Thus, all nodes in the subgraph $H'$ have degree at least 2 implying that $B'\subseteq B_\rho$. \qedhere
\end{enumerate}	
\end{proof}

\begin{theorem}\label{thm:characterising_nonHamiltonian_bases}
  Consider a subgraph $H = (\nodeset,B)$, with connected components $H_1,\dotsc,H_m$, $m \in \{ 1, \dotsc, \lfloor(n-1)/2\rfloor \}$. The arc set
  $B$ is a non-Hamiltonian basis of the polytope $\hBeta(G)$ if and only if the following three conditions
  are satisfied:
  \begin{enumerate}[(i)]
  \item $\lvert B_k\rvert=
    \begin{cases}
      \lvert\nodeset_1\rvert+1&\mbox{for } k=1,\\
      \lvert\nodeset_k\rvert  & \text{for }k = 2,\ldots,m;\\
    \end{cases}$\label{condition_1}
  \item If $m>1$, $H_k$ does not contain a balanced oriented cycle, for $k=2,\dotsc,m$; \label{condition_2}
  \item The subgraph $\rho(H_1)$ is the union of a short cycle and a noose
    path.\label{condition_3}
  \end{enumerate}
\end{theorem}

\begin{proof}  Let $B$ be a non-Hamiltonian basis of the polytope
$\hBeta(G)$ corresponding to an extreme point $\vect{x}$, and let $A_B$ be the associated submatrix
of $A$. Since $B$ is a basis, $A_B$ is invertible. Consequently,
conditions~\eqref{condition_1} and \eqref{condition_2} follow from lemmas~\ref{lem:zero_arcs} and
\ref{lem:good_augmented_trees}, respectively. Now, consider the subgraph $\rho(H_1)$. By
Lemma~\ref{lem:zero_arcs}, we know that the support $\supportgraph(G,\vect{x})$ is a subgraph of
$\rho(H_1)$. Moreover, as we remove an equal number of nodes and arcs to construct $\rho(H_1)$,
condition~\eqref{condition_1} implies that $\lvert B_\rho\rvert = \lvert V_\rho\rvert + 1$. It
follows that the support $\supportgraph(G,\vect{x})$ is identical to the graph $\rho(H_1)$. Hence,
by Theorem~\ref{thm:Ejov}, we can conclude condition~\eqref{condition_3}, that is $\rho(H_1)$ is the
union of a short cycle and a noose path.
   	
Conversely, suppose that conditions~\eqref{condition_1}--\eqref{condition_3} are satisfied. In order
to prove that $B$ is a non-Hamiltonian basis, we should only show that conditions \textbf{B1} and
\textbf{B2} hold, implying that $B$ is a feasible basis. Then, condition~\eqref{condition_3} with
Theorem~\ref{thm:Ejov} prescribe that $B$ is a non-Hamiltonian basis. If $A_{B_k}$ denotes a
submatrix of $A$ corresponding to the node and arc sets $V_k$ and $B_k$, respectively, for
$k=1,\ldots,m$, condition~\eqref{condition_1} implies that each $A_{B_k}$ is a square matrix and so
$A_B$ has a block diagonal structure. Hence, we have $\det(A_B) = \Pi_{k=1}^m\det(A_{B_k})$ and
$A_B$ is invertible if $A_{B_k}$ is invertible for all $k=1,\dots,m$. Consequently, to show that
$A_B$ is invertible, we need to show that each submatrix $A_{B_k}$ is
invertible. Lemma~\ref{lem:good_augmented_trees} with
conditions~\eqref{condition_1}--\eqref{condition_2} imply that $A_{B_k}$ is invertible for
$k = 2,\dotsc, m$. It remains to show that $A_{B_1}$ is invertible as
well. Let $\vect{\lambda}\in\reals^{\lvert B_1\rvert}$ be a vector satisfying
$ A_{B_1}\vect{\lambda} = \vect{0}$. As in the proof of Lemma~\ref{lem:good_augmented_trees}, we
show that $\lambda_{ij} = 0$ for all $(i,j) \in B_1 \setminus B_\rho$. From Theorem~\ref{thm:Ejov}
and condition~\eqref{condition_3}, the columns of $A_{B_1}$ corresponding to arcs in $B_\rho$ are
linearly independent; hence $\lambda_{ij} = 0$ for for all $(i,j) \in B_\rho$. Consequently, we have
$\vect{\lambda} = \vect{0} $, implying that the columns of $A_{B_1}$ are linearly independent. So,
condition \textbf{B1} holds. Finally, Theorem~\ref{thm:Ejov} and condition~\eqref{condition_3}
confirm that condition \textbf{B2} holds as well.
\end{proof}
\begin{figure}
  \centering
  \begin{subfigure}{.3\textwidth}
    \centering
    \begin{tikzpicture}[xscale=1.5,every node/.style={draw,shape=circle,outer sep=1pt,inner
        sep=1pt,minimum size=.15cm}]
		\node (1) at (0,0) {1};
		\node (2) at (1,0) {2};
		\node (3) at (1.5,-1) {3};
		\node (4) at (2,-2) {4};
		\node (5) at (1,-3) {5};
		\node (6) at (0,-2.5) {6};
		\node (7) at (-.5,-1) {7};

		\draw[thick,->] (1) -- (2);
		\draw[thick,->] (2) -- (3);
		\draw[thick,->] (3) -- (1);
		\draw[thick,bend right=40,->] (3) to (2);
		\draw[thick,->] (1) -- (7);
		\draw[thick,->] (1) -- (6);
		\draw[thick,->] (5) to (4);
		\draw[thick,->] (3) -- (4);
              \end{tikzpicture}
              \caption{}
            \end{subfigure}
            \hspace*{\fill}
            \begin{subfigure}{.3\textwidth}
		\centering
		\begin{tikzpicture}[xscale=1.5,every node/.style={draw,shape=circle,outer sep=1pt,inner sep=1pt,minimum size=.15cm}]
		\node (1) at (0,0) {1};
		\node (2) at (1,0) {2};
		\node (3) at (1.5,-1) {3};
		\node (4) at (2,-2) {4};
		\node (5) at (1,-3) {5};
		\node (6) at (0,-2.5) {6};
		\node (7) at (-.5,-1) {7};

		\draw[thick,->] (1) -- (2);
		\draw[thick,->] (2) -- (3);
		\draw[thick,->] (3) -- (1);
		\draw[thick,bend right=40,->] (3) to (2);
		\draw[thick,->] (4) -- (5);
		\draw[thick,->] (5) -- (6);
		\draw[thick,->] (7) -- (6);
		\draw[thick,->] (7) -- (4);
		\end{tikzpicture}
		\caption{}
	\end{subfigure}
	\hspace*{\fill}
	\begin{subfigure}{.3\textwidth}
		\centering
		\begin{tikzpicture}[xscale=1.5,every node/.style={draw,shape=circle,outer sep=1pt,inner sep=1pt,minimum size=.15cm}]
		\node (1) at (0,0) {1};
		\node (2) at (1,0) {2};
		\node (3) at (1.5,-1) {3};
		\node (4) at (2,-2) {4};
		\node (5) at (1,-3) {5};
		\node (6) at (0,-2.5) {6};
		\node (7) at (-.5,-1) {7};

		\draw[thick,->] (1) -- (2);
		\draw[thick,->] (2) -- (3);
		\draw[thick,->] (3) -- (1);
		\draw[thick,bend right=40,->] (3) to (2);
		\draw[thick,->] (4) -- (5);
		\draw[thick,bend left=30,->] (5) to (4);
		\draw[thick,bend left=30,->] (6) to (7);
		\draw[thick,->] (7) -- (6);
		\end{tikzpicture}
		\caption{}
	\end{subfigure}
	\caption{Three non-Hamiltonian bases of Type 4 corresponding to the support in Figure~\ref{fig:nonHam_EP}(\subref{fig:nonHamilton_Ext_2})}\label{fig:nonHamilton_basis_3}
      \end{figure}
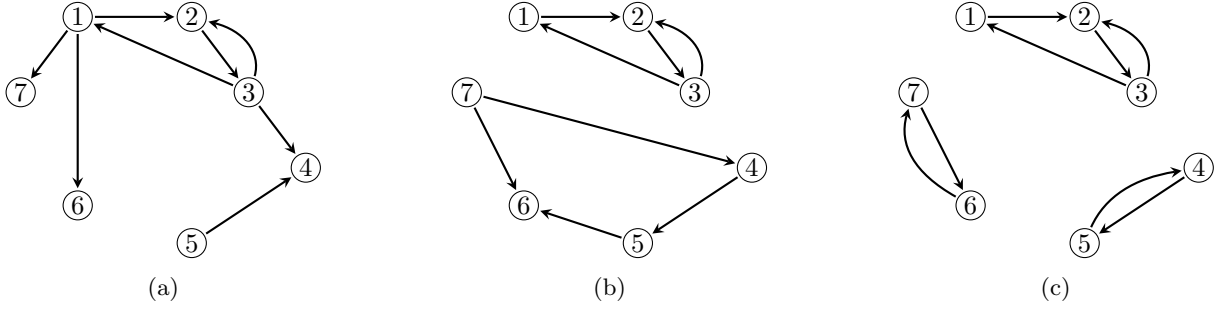

\begin{remark}\label{remark:feasible-bases}
  Theorem~\ref{thm:characterising_nonHamiltonian_bases} implies that the arc set of a spanning
  subgraph $H$ of $G$ is a non-Hamiltonian basis if and only if (i) $H$ has exactly
  $n+ 1$ arcs, (ii) every node in $H$ has positive degree, (iii) $H$ does not contain
  a balanced oriented cycle, and (iv) $H$ contains a short cycle and a noose path.
\end{remark}
\begin{remark}
	Theorem~\ref{thm:characterising_nonHamiltonian_bases} describes the structure of all non-Hamiltonian bases of types 1--4. However, as the supports of non-degenerate non-Hamiltonian extreme points of types 1--3 have exactly $n+1$ arcs, for their corresponding feasible bases, we have $H=H_1$ and the result of Theorem~\ref{thm:characterising_nonHamiltonian_bases} will be consistent with Theorem~\ref{thm:Ejov}. So, the main contribution of Theorem~\ref{thm:characterising_nonHamiltonian_bases} can be regarded for the structure of degenerate non-Hamiltonian bases of Type 4, which is utilized in Section~\ref{section:prevelance} to derive their expected prevalence in the polytope $\hBeta(G)$ when the input graph $G$ is random.
\end{remark}

As Figure~\ref{fig:nonHam_EP}(\subref{fig:nonHamilton_basis_1}) contains
a balanced oriented cycle $\overrightarrow{C}=\{(4,5),(6,5),(7,6),(7,4)\}$, it is now readily seen that it is not associated with a
feasible basis of the polytope $\hBeta(K_7)$. However, as
Figure~\ref{fig:nonHam_EP}(\subref{fig:nonHamilton_basis_2}) satisfies all conditions provided in
Theorem~\ref{thm:characterising_nonHamiltonian_bases}, it corresponds to a non-Hamiltonian basis of
Type 4. Figure~\ref{fig:nonHamilton_basis_3} shows three more non-Hamiltonian bases corresponding to
the support displayed in Figure~\ref{fig:nonHam_EP}(\subref{fig:nonHamilton_Ext_2}).


\section{Expected numbers of feasible bases of $\hBeta$-polytope for random graphs.}\label{section:prevelance}
As discussed in Section~\ref{section:formulation}, the correspondence between the sets of
Hamiltonian cycles in the graph $G$ and feasible bases of the polytope $\hBeta(G)$ can be exploited
to develop an algorithm that searches for Hamiltonian cycles. One key issue influencing the efficiency
of such an algorithm is the existence of a sufficiently large number of Hamiltonian
bases. More precisely, we require the ratio of the number of Hamiltonian bases to
non-Hamiltonian bases to be bounded below by $1/h(n)$, where $h(n)$ is a polynomial in
$n$. Then for a Hamiltonian input graph $G$ the probability of not finding a Hamiltonian
basis in the first $\alpha h(n)$ iterations would be expected to decay rapidly with $\alpha$ if the search algorithm
sampled feasible bases uniformly at random. For instance, if Algorithm~\ref{alg:HC_search} were
modified to search on feasible bases (instead of extreme points), according to a uniform
distribution, then we could expect to find a Hamiltonian cycle, with very high probability, in
polynomial time. Thus it is important to investigate the prevalence of Hamiltonian bases in
$\hBeta (G)$.

In this section, we exploit the structural results from Section~\ref{section:geometric_proerties} to
derive the expected prevalence of each of the five types of feasible bases of the polytope $\hBeta(G)$ where $G$ is a binomial random
graph and show that the majority of these are non-Hamiltonian bases of Type $4$. Thus, the
expected number of Hamiltonian bases of this polytope is exponentially small, when compared
to the expected number of non-Hamiltonian bases of Type $4$. Fortunately, in
\cite{eshragh2011hamiltonian}, it was shown that these undesirable feasible bases can be eliminated
by the addition of $2n-2$ lower and upper bound constraints.

The following definition is needed for the ensuing discussion. For an introduction to the theory
of general random graphs, we refer the interested reader to \citet{bollobas1998random}.

\begin{definition}	
  For a positive integer $n$ and a fixed probability $p\in [0,1]$, the \emph{directed binomial
    random graph} on $n$ nodes, denoted by $G_{n,p}$, is a random directed graph on $n$ nodes where
  each ordered pair of distinct nodes is, independently, connected by an arc with probability $p$.
\end{definition}

As discussed in Section~\ref{section:geometric_proerties}, the feasible bases of the polytope
$\hBeta(G)$ can be identified by certain subgraphs of $G$ on $n$ nodes and $n+1$ arcs. For
$k=0,1,2,3,4$, let $f_k(n)$ denote the number of subgraphs of the complete graph $K_n$ corresponding
to the feasible bases of Type $k$ of the polytope $\hBeta(K_n)$. Recall that the feasible bases
of Type $0$ are referred as Hamiltonian bases.
Label all $f_k(n)$ feasible bases of Type $k$ of the polytope $\hBeta(K_n)$,
$1,\dots,f_k(n)$. Define the binary random variable $I_{i}(k)$ that takes the value of one, if the
$i^{th}$ feasible basis of Type $k$ of the polytope $\hBeta(K_n)$ appears in the random polytope
$\hBeta(G_{n,p})$, and otherwise zero. Clearly,
\[\expect {I_i(k)}= p^{n+1},\]
where $\textbf{E}$ denotes the expected value. Although $I_i(k)$ are dependent random variables, as the expectation is a linear operator, we have  
\begin{align}\label{EquExp}
 \expect{\sum_{i=1}^{f_k(n)} I_i(k)}  = \sum_{i=1}^{f_k(n)}\expect {I_i(k)}=f_k(n) p^{n+1}.
\end{align}
Equation~\eqref{EquExp} implies that the expected number of feasible bases of Type $k$ of the polytope $\hBeta(G_{n,p})$ equals $f_k(n) p^{n+1}$ for $k=0,\dots,4$. In particular, the
ratio between the expected number of Hamiltonian bases and the expected total number of
feasible bases in the random polytope $\hBeta(G_{n,p})$ is independent of $p$ and equals
$f_0(n)/[f_0(n)+\dotsb+f_4(n)]$. Theorem~\ref{thm:Expected_Num_Ext} provides the expected number of
feasible bases of types 0--3 as well as a lower bound on the expected number of feasible bases of
Type 4 by determining the value of $f_k(n)$ for $k=0,\dots,3$ and finding a lower bound on $f_4(n)$.
\begin{theorem}\label{thm:Expected_Num_Ext}
  Consider the binomial random graph $G_{n,p}$ and the corresponding polytope $\hBeta(G_{n,p})$. The expected number of 
  \begin{enumerate}[(i)]
  \item\label{item:expectation_type_0}Hamiltonian bases is  $\displaystyle(n-2)n!p^{n+1}$;\medskip
  \item\label{item:expectation_type_1}non-Hamiltonian bases of Type 1 is $\displaystyle\frac12(n-3)n!p^{n+1}$;\medskip
  \item\label{item:expectation_type_2}non-Hamiltonian bases of Type 2 is  $\displaystyle\frac16(n-4)(n-3)(n+1)(n-1)!p^{n+1}$;\medskip
  \item\label{item:expectation_type_3}non-Hamiltonian bases of Type 3 is $\displaystyle\frac16(n-2)(n-1)n!p^{n+1}$;\medskip
  \item\label{item:expectation_type_4}non-Hamiltonian bases of Type
    $4$ is at least
    $\displaystyle(n-1)(n-2)(n-3)^{n-5}2^{n-4}p^{n+1}$.
  \end{enumerate}
\end{theorem}
\begin{proof}\hfill
  \begin{enumerate}[(i)]
  \item The complete graph $K_n$ possesses $n(n-1)$ arcs and $(n-1)!$ Hamiltonian cycles. By Proposition~\ref{prop:Hamilton_basis} and Corollary~\ref{lem:adj_HExt}, for each Hamiltonian cycle, any of the $n(n-2)$ remaining arcs can be added to form a feasible basis. It follows that $f_0(n)=(n-2)n!$.
  \item A non-Hamiltonian basis of Type 1 consists of a short cycle of length $k \in \{ 2,3,\dotsc,n-2 \} $,
    a noose cycle of length $n-k$, and an arc joining them. For a  fixed $k$, there are 
    \begin{itemize}
    	\item $\binom{n-1}{k-1}(k-1)!$ ways to construct a short cycle;
    	
    	\item $(n-k-1)!$ ways to construct a noose cycle from the remaining nodes, after fixing the short cycle; 
    	
    	\item $\binom{k}{1}\binom{n-k}{1}$ options to choose a node from the short cycle and the noose cycle to connect them together. 
    \end{itemize}
    Consequently,
    \[ f_1(n)=\sum_{k=2}^{n-2}\frac{(n-1)!}{(n-k)!}(n-k-1)!k(n-k) =(n-1)!\sum_{k=2}^{n-2}k=\frac12(n-3)n!.\]
  \item A non-Hamiltonian basis of Type 2 consists of a short cycle of length $k\in\{2,3,\dotsc,n-3\}$, a
    noose cycle of length $l\in\{2,3,\dotsc,n-k-1\}$, and a path of length $n-(k+l)+1$ joining
    them. For fixed $k$ and $l$, there are 
    \begin{itemize}
    	\item $\binom{n-1}{k-1}(k-1)!$ ways to construct a short cycle;
    	
    	\item $\binom{n-k}{l}(l-1)!$ ways to construct a noose cycle, after fixing the short cycle; 
    	
    	\item $(n-k-l)!$ ways to permute the remaining nodes to construct the connecting path from the short cycle to the noose cycle; 
    	
    	\item $\binom{k}{1}\binom{l}{1}$ options to choose a node from the short cycle and a node
          from the noose cycle to connect them through the path constructed in (iii). 
    \end{itemize}
    Consequently,
    \begin{multline*}
      f_2(n) =
      \sum_{k=2}^{n-3}\sum_{l=2}^{n-k-1}\frac{(n-1)!}{(n-k)!}.\frac{(n-k)!}{l(n-k-l)!}(n-k-l)!kl
      =\sum_{k=2}^{n-3}\sum_{l=2}^{n-k-1}k(n-1)!=(n-1)!\sum_{k=2}^{n-3}(n-k-2)k\\= (n-1)!\left((n-2)\sum_{k=2}^{n-3}k - \sum_{k=2}^{n-3}k^2\right)
      = \frac 16(n-4)(n-3)(n+1)(n-1)!.
   \end{multline*}    
 \item For a non-Hamiltonian basis of Type 3 with a short cycle of length $k\in\{2,3,\dotsc,n-1\}$, there
   are
   \begin{itemize}
   \item $\binom{n-1}{k-1}(k-1)!$ ways to construct a short cycle;
   
   \item $(n-k)!$ ways to permute the remaining nodes, after fixing the short cycle; 
   
   \item $\binom{k-1}{1}$ options to choose a splitting node on the short cycle; 
   
   \item $\binom{l}{1}$ options to choose a node with in-degree 2 on  the short cycle, where $l$ is the location of the splitting node on the short cycle starting from node 1.
   \end{itemize}
   Consequently,
   \[f_3(n)= \sum_{k=2}^{n-1}\sum_{l=1}^{k-1}\frac{(n-1)!}{(n-k)!}(n-k)!l=(n-1)!\sum_{k=2}^{n-1} \frac{k(k-1)}{2}=\frac16(n-2)(n-1)n!.
   \] 
 \item In order to establish a lower bound for $f_4(n)$, following
   Remark~\ref{rem:support_of_type4}, we bound the number of feasible bases of Type 3 corresponding
   to non-Hamiltonian extreme point $\vect x$ such that their supports have only three nodes
   including one short cycle and one noose cycle, each of length 2. If $\vect x$ is
   non-Hamiltonian extreme point of the polytope $\hBeta(K_n)$ with this structure, there are
   $\binom{n-1}{1}\binom{n-2}{1}=(n-1)(n-2)$ ways to construct the support
   $\supportgraph(K_n,\vect{x})$. Any undirected tree spanning the remaining $n-3$ nodes can be
   oriented in $2^{n-4}$ ways, and for each of these orientations there is always an additional arc that can be added without creating a balanced oriented cycle. As a consequence of
   Theorem~\ref{thm:characterising_nonHamiltonian_bases} and Remark~\ref{remark:feasible-bases}, the
   result is a feasible basis of Type 4 corresponding to $\vect x$. By Cayley's formula (see, for
   instance, \cite[Chapter~30]{aigner2010proofs}), the number of spanning trees on $n-3$ labeled
   nodes is equal to $(n-3)^{n-5}$. Thus, we have
   \[ f_4(n)\geq(n-1)(n-2)(n-3)^{n-5}2^{n-4}.\qedhere\]
 \end{enumerate}
\end{proof}
\begin{corollary}\label{cor:ratio_nHExt4_totalExt}
	In the polytope $\hBeta(G_{n,p})$, for sufficiently large $n$, we have
	\begin{enumerate}[(i)]
		\item \[\frac{\expect{\text{Number of feasible bases of Type 4}}}{\expect{\text{Total number of feasible
				bases}}}\geq 1-\frac{n^{11/2}}{e^{n}2^{n-9}};\]
\item
	\[\frac{\expect{\text{Number of Hamiltonian bases}}}{\expect{\text{Total number of feasible
				bases}}}\leq\frac{n^{9/2}}{e^{n-1}2^{n-9}}.\]
	\end{enumerate}
\end{corollary}
\begin{proof}\hfill
  \begin{enumerate}[(i)]
  	\item The quotient on the left-hand side is
  \begin{multline*}
    \frac{f_4(n)}{f_0(n)+f_1(n)+f_2(n)+f_3(n)+f_4(n)}=1-\frac{f_0(n)+f_1(n)+f_2(n)+f_3(n)}{f_0(n)+f_1(n)+f_2(n)+f_3(n)+f_4(n)}\\
                                                     \geq 1-\frac{f_0(n)+f_1(n)+f_2(n)+f_3(n)}{f_4(n)}.
  \end{multline*}
  Using Stirling's formula to bound the factorials in parts~(\ref{item:expectation_type_0})--(\ref{item:expectation_type_3}) of Theorem~\ref{thm:Expected_Num_Ext}, 
  \[f_0(n)+f_1(n)+f_2(n)+f_3(n)= \frac1 3 (n^3-7n+6)(n-1)! \leq\frac 1 3 n^2n!\leq n^{n+5/2}e^{-n}.\]
  Furthermore, for sufficiently large $n$, we have
  \begin{multline}
    f_4(n)\geq(n-1)(n-2)(n-3)^{n-5}2^{n-4}=n^{n-3}\left(\frac{n-1}{n}\right)\left(\frac{n-2}{n}\right)\left(\frac{n-3}{n}\right)^{n-5}2^{n-4}\\
    =n^{n-3}2^{n-4}\left(1-\frac1n\right)\left(1-\frac2n\right)\left(\left(1-\frac3n\right)^{n/3}\right)^{3(n-5)/n}\\\geq n^{n-3}2^{n-4}\left(1-\frac2n\right)^2 \left(\frac1e\right)^3
    \geq n^{n-3}2^{n-4}\left(\frac{27}{32}\right)\left(\frac13\right)^{3}=
    n^{n-3}2^{n-9}.\label{eq:f4_lower_bound}
  \end{multline}	
  Thus, we obtain
  \[
  \frac{f_4(n)}{f_0(n)+f_1(n)+f_2(n)+f_3(n)+f_4(n)}
  \geq 1-\frac{n^{n+5/2}e^{-n}}{n^{n-3} 2^{n-9}} = 1-\frac{n^{11/2}}{e^{n} 2^{n-9}}.
  \]
\item From part (i) of Theorem~\ref{thm:Expected_Num_Ext} and Inequality~\eqref{eq:f4_lower_bound}
  as well as utilizing Stirling's formula, the left-hand side is bounded above as follows:
 \[ \frac{f_0(n)}{f_0(n)+f_1(n)+f_2(n)+f_3(n)+f_4(n)}  \leq \frac{f_0(n)}{f_4(n)} \leq \frac{(n-2)n!}{n^{n-3}2^{n-9}}\leq \frac{n^{9/2}}{e^{n-1} 2^{n-9}}. \qedhere\] 
\end{enumerate}
\end{proof}
\begin{remark}\label{rem:expected_num_type4}
  Corollary~\ref{cor:ratio_nHExt4_totalExt} shows that even the expected number of feasible
  bases of Type 4 corresponding to a portion of supports on three nodes is exponentially larger than
  the expected number of feasible bases of types 0, 1, 2 and 3 all together. Thus, non-Hamiltonian bases of Type 4 can make up the majority of feasible bases of the polytope $\hBeta(G)$. This result is supported numerically in Section~\ref{section:numerical_result}.
\end{remark}


\section{Reducing the feasible region.}\label{section:reducing_region}
Theorem~\ref{thm:Expected_Num_Ext} and Corollary~\ref{cor:ratio_nHExt4_totalExt} show that the main
obstacle to finding a Hamiltonian basis of the polytope $\hBeta(G)$ may lie in the
abundance of non-Hamiltonian bases of Type $4$. As indicated by
Corollary~\ref{cor:ratio_nHExt4_totalExt}\,, as the number of nodes of the graph $G$, $n$,
increases, the ratio between the numbers of non-Hamiltonian bases of Type $4$ and all
feasible bases of the polytope $\hBeta(G)$ could tend to one exponentially fast. It is therefore of interest to modify the polytope $\hBeta(G)$ in
such a way that all non-Hamiltonian bases of Type $4$ are eliminated, while all Hamiltonian
feasible bases are preserved. To this end, \citet{eshragh2011hamiltonian} added $2(n-1)$ of the
so-called wedge constraints \eqref{wedge_constraints}, to $\hBeta(G)$, thereby defining a modified
polytope $\whBeta(G)$ with the following constraints set:
\begin{align}
\vect{x} \in \hBeta(G),\nonumber\\
 \beta^{n-1} \leq \sum_{j \in \outi} x_{ij} &\leq \beta  && \mbox{for all } i \in \nodeset \setminus \{1\}.  \label{wedge_constraints}
\end{align}
In \cite{eshragh2011hamiltonian} it was shown that, when $\beta$ is sufficiently close to one, the
wedge constraints~\eqref{wedge_constraints} excise what is, typically, the majority of
non-Hamiltonian bases. This result is summarized in Theorem~\ref{thm:cutting_off_nonH_Ext}.

\begin{theorem}[\citet{eshragh2011hamiltonian}]\label{thm:cutting_off_nonH_Ext}
  Consider the graph $G=(\nodeset,\arcset)$ on n nodes and the corresponding polytopes $\hBeta(G)$ and $\whBeta(G)$. For values of $\beta\in\left((1-\frac{1}{n-2})^{\frac{1}{n-2}},1\right)$, the only possible common extreme points of polytopes $\hBeta(G)$ and $\whBeta(G)$ are the Hamiltonian extreme points and the non-Hamiltonian extreme points of Type 1.
\end{theorem}

Our results in Section~\ref{section:prevelance} show that, for a given binomial random
graph, the majority of non-Hamiltonian bases are of Type $4$. Since the result of Theorem~\ref{thm:cutting_off_nonH_Ext} can clearly be extended from extreme
  points to feasible bases, removing
non-Hamiltonian extreme points of Type 4 may be interpreted as the efficiency of the wedge
constraints. However, the wedge constraints may also add new extreme points. In order to fully
investigate the efficiency of the wedge constraints, we need to determine the structure and
prevalence of feasible bases of these new extreme points of the polytope $\whBeta(G)$.
Analogous to \citet{eshragh2011hybrid}, we introduce the concept of quasi-Hamiltonian bases
for $\whBeta(G)$ as follows.

\begin{definition}
  Let $\vect{x} =(x_{ij})\in\reals^{\lvert E\rvert}$ be an extreme point of $\whBeta(G)$. The
  extreme point $\vect x$ is called \emph{quasi-Hamiltonian} if any path
  $i_1 \rightarrow i_2 \rightarrow \dots\rightarrow i_n \rightarrow i_{n+1}$, where
  $i_{k+1} \in \arg\max_{j \in \mathcal{N^+}(i_k)}\{ x_{i_k j} \}$ for $k = 1,\dots, n$ is a
  Hamiltonian cycle in $G$. In this case, any feasible basis corresponding to $\vect x$ is called a
  \emph{quasi-Hamiltonian basis}.
\end{definition}
As an example, consider the following quasi-Hamiltonian extreme point of the polytope $\whBeta(K_6)$ for $\beta=0.999$:
\begin{multline*}
	 x_{12} = 1.00000,\ x_{26} = 0.999000,\ x_{34} = 0.994013;\ x_{36} = 0.000997,\ x_{45} =
         0.995010,\   
	 x_{51} = 0.995010,\\ x_{54} = 0.001993,\ x_{63} = 0.996006,\ x_{65} = 0.00299101, 
\end{multline*}
and all the other $x_{ij}$ are zero. We can easily see that this extreme point 
corresponds to the Hamiltonian cycle
$1\rightarrow 2\rightarrow 6\rightarrow 3\rightarrow 4\rightarrow 5\rightarrow 1$. Any feasible
basis of the polytope $\mathcal{WH}_\beta(K_6)$ corresponding to this particular extreme point is
called a quasi-Hamiltonian basis.

\begin{remark}
  Evidently, every Hamiltonian extreme point in the polytope $\hBeta(G)$ is also a quasi-Hamiltonian
  extreme point in the polytope $\whBeta(G)$. Furthermore, quasi-Hamiltonian bases increase the pool
  of desirable feasible bases.
\end{remark}

In the next section, we develop two algorithms to study the effectiveness of the wedge constraints
in removing non-Hamiltonian bases and the role of the discount parameter $\beta$. We present
promising numerical results regarding the efficiency of the modified polytope $\whBeta(G)$ over
$\hBeta(G)$, in the sense that the prevalence of quasi-Hamiltonian bases among all feasible bases of
$\whBeta(G)$ appears much greater than the prevalence of Hamiltonian bases among all feasible bases
of $\hBeta(G)$. This means that, based on our numerical evidence, the wedge constraints do not add
an excessive number of non-Hamiltonian bases.


\section{Computational results.}\label{section:numerical_result}
The polytopes $\hBeta(G)$ and $\whBeta(G)$ introduced in sections~\ref{section:formulation} and~\ref{section:reducing_region}, respectively, can be used as the basis of a search algorithm for HCP
(Algorithm~\ref{alg:HC_search}). As indicated earlier, one key improvement to enhance the efficiency of such an
algorithm would be to increase the ratio of Hamiltonian bases to all feasible bases in the
polytopes. Corollary~\ref{cor:ratio_nHExt4_totalExt} indicates that, on average, the vast majority of
feasible bases of the polytope $\hBeta(G_{n,p})$ are non-Hamiltonian bases of Type 4 and therefore the number of
Hamiltonian bases is not sufficiently large for our
purposes. Theorem~\ref{thm:cutting_off_nonH_Ext} show
that the wedge constraints cut off three types of non-Hamiltonian extreme points of $\hBeta(G)$. However, they may
also introduce a number of additional extreme points and consequently new feasible bases to the polytope $\whBeta(G)$. In order to compare the polytopes $\hBeta(G)$ and
$\whBeta(G)$, we develop two random walk algorithms on feasible bases of the two polytopes.

A random walk algorithm to find one Hamiltonian cycle is Algorithm~\ref{alg:role_of_beta}. More
precisely, this algorithm runs a simple uniform random walk on the feasible bases of the polytope
$\whBeta(G)$ for an input Hamiltonian graph $G$ and stops when it finds a quasi-Hamiltonian basis
and reports the total number of steps to reach that basis.
\begin{algorithm*}
  \caption{}\label{alg:role_of_beta}
  \begin{algorithmic}[1]    
    \State \textbf{Input:} A graph $G$ and a positive integer MaxStep
    \State Step $\leftarrow$ 0
    \State Let $B$ be an initial feasible basis for the polytope $\whBeta(G)$
    \While{$B$ is not quasi-Hamiltonian}
    \State Step $\leftarrow$ Step+1
    \State $B\ \leftarrow$ a basis chosen uniformly at random from the feasible bases adjacent to $B$ 
    \EndWhile
    \State \textbf{Output:} Step (the number of steps before finding a quasi-Hamiltonian feasible basis)
  \end{algorithmic}
\end{algorithm*}

The random walk step in Line 6 of Algorithm~\ref{alg:role_of_beta} is implemented as follows. The
bases adjacent to $B$ are given by pairs of a leaving variable and an entering variable chosen from the set of basic variables and non-basic variables, respectively. Each of
these pairs is checked for feasibility, and the feasible ones are listed together as the set of feasible bases adjacent to $B$. Finally, one of the
adjacent feasible bases is chosen uniformly at random.

We implemented Algorithm~\ref{alg:role_of_beta} in Matlab R2015b and tested it on a range of Hamiltonian binomial
random graphs $\bar{G}_{n,p}$. The latter are merely binomial random graphs $G_{n,p}$ augmented by 
insertion (if necessary) of arcs corresponding to one randomly generated Hamiltonian cycle, thereby
ensuring the Hamiltonicity of $\bar{G}_{n,p}$. In Table~\ref{table:No.of_iteration_findHC}, we report
the performance of Algorithm~\ref{alg:role_of_beta} on several of these graphs with
$p = 3/n$, $n = 10,\dotsc,70$ and $\beta = 0.999$.

\begin{table}	
\centering
\begin{minipage}{.6\textwidth}
	\centering
\caption{No. of steps required to find one Hamiltonian cycle}\label{table:No.of_iteration_findHC}
  \begin{tabular}{ c r }\toprule  	
    Random Graph &  No. of Steps \\  \midrule
    $\bar{G}_{10,3/10}$  & $6$ \\
    $\bar{G}_{20,3/20}$  &  $454$\\
    $\bar{G}_{30,3/30}$  & $1358$ \\
    $\bar{G}_{40,3/40}$  & $5714$\\
    $\bar{G}_{50,3/50}$  & $4392$\\
    $\bar{G}_{60,3/60}$  & $6230$\\
    $\bar{G}_{70,3/70}$  & $6387$\\ \bottomrule
  \end{tabular} 
	\end{minipage}
\end{table}
It is important to note the slow growth of the required iterations with increases in $n$. For
NP-complete problems, such as HCP, doubling or tripling $n$ would normally dramatically increase
the number of iterations. Furthermore, we emphasize the crucial role of the wedge constraint and the
parameter $\beta$.

We modified Algorithm~\ref{alg:role_of_beta} slightly by replacing $\whBeta(G)$ and
quasi-Hamiltonian with
$\hBeta(G)$ and Hamiltonian, respectively, and running it on the same $\bar{G}_{n,p}$ graphs reported in
Table~\ref{table:No.of_iteration_findHC}. In all cases, we failed to find Hamiltonian cycles in
$1,000,000$ steps. Furthermore, all one million of the non-Hamiltonian basis
identified in the process were of Type 4. This is consistent with
Corollary~\ref{cor:ratio_nHExt4_totalExt} and Remark~\ref{rem:expected_num_type4}.

Another interesting feature of the polytope $\whBeta(G)$ is the role played by the parameter
$\beta$. Theorem~\ref{thm:cutting_off_nonH_Ext} implies that the wedge constraints can be efficient
when $\beta$ is sufficiently close to one. In order to test this result numerically, we implemented
Algorithm~\ref{alg:role_of_beta} on $\bar{G}_{30,3/30}$ for various values of $\beta$. We
were interested in determining how increases in the parameter $\beta$ would affect the search for
quasi-Hamiltonian bases. The results are given in Table~\ref{table:role_of_beta}. The first
column contains the chosen values of $\beta$, while the second column shows the number of steps
that the algorithm took to find a quasi-Hamiltonian basis. When
Algorithm~\ref{alg:role_of_beta} did not find a quasi-Hamiltonian basis in $30,000$
steps, we reported `fail'. Table~\ref{table:role_of_beta} demonstrates the crucial role of $\beta$ in the polytope $\whBeta(G)$ where fewer steps are required to find a quasi-Hamiltonian  basis for larger values of $\beta$. These numerical experiments are aligned with theoretical results stated in Theorem~\ref{thm:cutting_off_nonH_Ext}.
\begin{remark}
  Analogous to the proof of Theorem 4.3 in \cite{eshragh2011hamiltonian}, we can show that if a
  given graph $G$ is Hamiltonian, then Algorithm~\ref{alg:role_of_beta} almost surely converges to a
  quasi-Hamiltonian basis in a finite number of steps.
\end{remark}

Next we wished to demonstrate, numerically, the prevalence of quasi-Hamiltonian bases in the random polytope 
$\whBeta(\bar{G}_{n,p})$. Hence, Algorithm~\ref{alg:role_of_beta} has been modified as described in Algorithm~\ref{alg:wedge_efficiency}.
\begin{algorithm}[htb]
	\caption{}\label{alg:wedge_efficiency}
	\begin{algorithmic}[1]
		\State \textbf{Input:} A graph $G$ and a positive integer MaxStep		
		\State Step $\leftarrow$ 0
		\State Counter $\leftarrow$ 0
		\State $B\ \leftarrow$ an initial feasible basis for the polytope $\whBeta(G)$
		\While{Step $\leq$ MaxStep}
		\State $B\ \leftarrow$ a basis chosen uniformly at random from the feasible bases adjacent to $B$
		\If { the current feasible basis is a quasi-Hamiltonian}
		\State Counter $\leftarrow$ Counter$+1$
		\EndIf
		\State Step $\leftarrow$ Step$+1$
		\EndWhile
                \State \textbf{Output:} Counter (the number of visited quasi-Hamiltonian feasible bases)
	\end{algorithmic}
\end{algorithm}

Thus, Algorithm~\ref{alg:wedge_efficiency} counts the number of quasi-Hamiltonian bases
visited during a prescribed number of steps, say MaxStep, for an input graph $G$. We ran this algorithm for
$100,000$ steps on the same seven $\bar{G}_{n,p}$ graphs as reported above and with
$\beta=0.999$. The results are summarized in Table~\ref{table:prevalence_q-Hfb}, where the second
column lists the number of visited feasible bases (out of $100,000$) that were
quasi-Hamiltonian. These numbers are remarkably high and appear to be a direct consequence of both
the nature of wedge constraints and the high value of $\beta$. This is consistent with
Theorem~\ref{thm:cutting_off_nonH_Ext} and suggests that many of the additional extreme points
created by wedge constraints may be quasi-Hamiltonian.

\begin{table}[htb]
  \begin{minipage}[t]{.4\textwidth}
  \centering  
  \caption{Dependence of Algorithm~\ref{alg:role_of_beta} on $\beta$ for the input graph  
  	$\bar{G}_{30,3/10}$}\label{table:role_of_beta}
  \begin{tabular}{c r} \toprule
    Value of $\beta$ & No. of Steps\\ \midrule
    $0.5$&fail\\
    $0.6$&fail\\
    $0.7$&fail\\
    $0.8$&fail\\
    $0.9$&fail\\
    $0.99$&$1944$\\
    $0.995$&$1139$\\
    $0.999$&$1001$\\
    $0.9995$&$84$\\ 
    $0.9999$&$23$\\ \bottomrule 
  \end{tabular}
\end{minipage} 
\hfill
\begin{minipage}[t]{.4\textwidth}
  \centering
  \caption{Number of quasi-Hamiltonian bases visited by Algorithm~\ref{alg:wedge_efficiency}}\label{table:prevalence_q-Hfb}
  \begin{tabular}{c r}\toprule
    Random Graph &  No. of Q-HBs\\ \midrule    
    $\bar{G}_{10,3/10}$  &  $70,197$\\
    $\bar{G}_{20,3/20}$  &  $47,897$\\
    $\bar{G}_{30,3/30}$  & $6629$ \\
    $\bar{G}_{40,3/40}$  & $34,434$\\
    $\bar{G}_{50,3/50}$  & $19,472$\\
    $\bar{G}_{60,3/60}$  & $1790$\\ 
    $\bar{G}_{70,3/70}$  & $2863$\\ \bottomrule
  \end{tabular}  
\end{minipage}
\hspace*{1cm}
\end{table}


\section{Conclusion}\label{sec:conclusion}
This paper should be seen as a continuation of results reported in \citet{feinberg2000constrained},
\citet{ejov2009refined}, \citet{eshragh2011hybrid} and \citet{eshragh2011hamiltonian}. The original
idea of exploiting the discounted occupational measures in MDP in which a graph $G$ is embedded led
to algorithmic insights only after the structure of the associated polytope $\hBeta(G)$ had been
examined in more detail. In particular, the analysis of extreme points of $\hBeta(G)$ presented in
\cite{ejov2009refined,eshragh2011hamiltonian} necessitated a more precise analysis of their
corresponding feasible bases (Proposition~\ref{prop:Hamilton_basis} and
Theorem~\ref{thm:characterising_nonHamiltonian_bases}). The subsequent analysis of the expected
prevalence of the five types of feasible bases in the random polytope $\hBeta(G_{n,p})$
(Theorem~\ref{thm:Expected_Num_Ext}) exposes the essential difficulty of HCP by demonstrating the
asymptotic dominance of non-Hamiltonian bases of Type 4
(Corollary~\ref{cor:ratio_nHExt4_totalExt}). On the positive side, it follows from
Theorem~\ref{thm:Expected_Num_Ext} that, asymptotically, the expected number of Hamiltonian bases is
twice that of non-Hamiltonian bases of Type 1 in the random polytope $\hBeta(G_{n,p})$. This is
important in view of Theorem~\ref{thm:cutting_off_nonH_Ext} (proved in
\cite{eshragh2011hamiltonian}), which suggests that adding the wedge constraints and thereby
converting $\hBeta(G)$ to $\whBeta(G)$ results in a polytope where searches for quasi-Hamiltonian
bases may be more effective. Recall that in Section~\ref{section:formulation} we noted that the
random walk based approach leads to an efficient algorithm provided (i) there are sufficiently many
feasible bases corresponding to Hamiltonian cycles, and (ii) the convergence of the random walk to
the uniform distribution is sufficiently fast. While Corollary~\ref{cor:ratio_nHExt4_totalExt} is
bad news regarding condition (i) for $\hBeta(G)$, the numerical results of Section~\ref{section:numerical_result} indicate
that the situation is much better for $\whBeta(G)$. More precisely, we make the following
conjecture on the feasible bases of the random polytope $\whBeta(\bar{G}_{n,p})$.
\begin{conjecture}
  There exist positive constants $c$, $\delta$ and $k$, such that for all values $\beta\in (1-e^{-cn},1)$,
  with high probability, the expected proportion of feasible bases of $\whBeta(\bar{G}_{n,p})$ that are
  quasi-Hamiltonian is at least $\delta/n^k$.
\end{conjecture}
This conjecture states that there are sufficiently many quasi-Hamiltonian feasible bases in the polytope $\whBeta(G)$ for large values of $\beta$. This statement was supported by numerical results reported in Section~\ref{section:numerical_result}. Proving this conjecture would be a big step towards an efficient variant of
Algorithm~\ref{alg:HC_search} by sampling feasible bases of $\whBeta(G)$ instead of extreme points of
$\hBeta(G)$.

\section*{Acknowledgments.} The authors are grateful to Professor Michael Saunders, an anonymous reviewer and the Area Editor for their invaluable comments that helped improve the previous version of this paper. The second author wishes to acknowledge the support from the Australian Research Council
under grants DP150100618 and DP160101236.

\bibliographystyle{plainnat} 
\bibliography{biblio}

\end{document}